\documentclass[10pt,onecolumn]{IEEEtran}
\ifCLASSINFOpdf
\else
\fi
\usepackage{epsfig}
\usepackage{amssymb}
\usepackage{amsmath}
\usepackage{color}
\usepackage{amsfonts}
\newtheorem{theorem}{Theorem}

\newtheorem{lemma}{Lemma}
\newtheorem{corollary}{Corollary}

\newtheorem{remark}{Remark}
\newtheorem{problem}{Problem}

\newtheorem{assumption}{Assumption}

\newtheorem{proof}{Proof}
\newtheorem{proof of Theorem 1}{Proof of Theorem 1}
\newcommand*{\QEDA}{\hfill\ensuremath{\blacksquare}}

\DeclareMathOperator{\diag}{diag}

\begin{document}
%

\title{Online distributed algorithms for seeking generalized Nash equilibria in dynamic environments}
%
%
%

\author{Kaihong~Lu,~
       Guangqi~Li~
       and~Long~Wang~
\thanks{This work was supported by National Natural Science Foundation of China (61751301 and 61533001). (\emph{Corresponding author: Long Wang})}
\thanks{K. Lu and G. Li are with the Center for Complex Systems, School of Mechano-electronic Engineering, Xidian University, Xi'an 710071, China
 (e-mail:khong\_lu@163.com; gqli@xidian.edu.cn)
 }
\thanks{L. Wang is with the Center for Systems and Control, College of Engineering, Peking
University, Beijing 100871, China
 (e-mail: longwang@pku.edu.cn)}
}

\maketitle

\begin{abstract}
In this paper, we study the distributed generalized Nash equilibrium seeking problem of non-cooperative games in dynamic environments. Each player in the game aims to minimize its own time-varying cost function subject to a local action set. The action sets of all players are coupled through a shared convex inequality constraint. Each player can only have access to its own cost
function, its own set constraint and a local block of the
inequality constraint, and can only communicate with its neighbours via a connected graph. Moreover, players do not have prior
knowledge of their future cost functions. To address this problem, an online distributed algorithm is proposed based on consensus algorithms and a primal-dual strategy. Performance of the algorithm is measured by using dynamic regrets. Under mild assumptions on graphs and cost functions, we prove that if the deviation of variational generalized Nash equilibrium sequence increases within a certain rate, then the regrets, as well as the violation of inequality constraint, grow sublinearly. A simulation is presented to demonstrate the effectiveness
of our theoretical results.

\end{abstract}


%
\IEEEpeerreviewmaketitle
\section{Introduction}
\IEEEPARstart{N}{ash} equilibrium (NE) seeking and generalized Nash equilibrium (GNE) seeking in non-cooperative games have received increasing attention in recent years. This is due to their wide practical applications in large-scale systems including sensor networks \cite{M. S. Stankovi2}, distributed power systems \cite{Paccagnan1}, and social networks \cite{Ghaderi2}.

Recently, distributed NEs or GNEs seeking in non-cooperative games without full action information have been considered \cite{F. Salehisadaghiani4}-\cite{{G. Dian4}}, \cite{S. Liang}-\cite{K. Lu 6}. In such problems, players cannot directly observe actions of players who are not their neighbors, but they can make decisions based on local estimations on actions of the others. In \cite{F. Salehisadaghiani4} and \cite{F. Salehisad2}, asynchronous
gossip-based algorithms are presented for seeking an NE. In \cite{M. Ye3} and \cite{M. Ye2}, continuous-time distributed NE seeking strategies are proposed by combining consensus algorithms and the gradient strategy. In \cite{{G. Dian4}}, an augmented gradient play dynamics is proposed by exploiting some incremental passivity properties of the pseudo-gradient mapping. In applications where players compete
for shared network resources, shared constraints that couple players' actions together usually appear \cite{ Paccagnan1, P. Yi 1}. In \cite{Grammatico11, Grammatico12}, semi-decentralized algorithms are proposed for seeking GNEs of aggregative games with shared coupling constraints, where a central
coordinator is required to broadcast the common multipliers and
aggregative variables. In \cite{S. Liang}, a distributed prime-dual strategy using partial decision
information is designed for searching a GNE in aggregate games with a shared linear equation constraint. In \cite{Pavel L11, P. Yi115}, for general non-cooperative games with shared linear constraints, a distributed algorithm via an operator-splitting approach is proposed for seeking the GNE under partial decision
information, and the asynchronous operator-splitting algorithm is studied in \cite{ P. Yi5}. Moreover, for non-cooperative games without full decision information,
a continuous-time distributed gradient projected algorithm is presented for seeking the GNE in \cite{K. Lu 6}.

If cost functions in a game are time-varying and they are only available to players after decisions are made, then the game is called an online game. Accordingly, an algorithm for dealing with such a game is referred to as an online algorithm or a learning algorithm. It is obvious that all of the aforementioned works consider offline games, and strategies in them are offline
algorithms. However, dynamic environments arise in many practical applications. For example, in the problem of allocating radio resources, due to the uncertainties in dynamic wireless environment, cost functions in the game is time-varying, sometimes changes can only be seen in hindsight \cite{M. Maskery1}. To adapt to such dynamic environments, it is necessary to develop online strategies. It is well-known that any online algorithm should mimic the performance of its offline counterpart, and the gap between them is the regret. For instance, in online optimization problems \cite{E. C. Hall}-\cite{A. Jadbabaie8}, the most stringent offline benchmark is to minimize the cost function at each time. The corresponding regret is called as dynamic regret. Using dynamic regrets, online optimization problems become insolvable in the worst case when cost functions fluctuate drastically. The difficulty can be characterized via a complexity measure that captures the variation in the optimal solution sequence \cite{E. C. Hall}. Different from optimization problems, since each player aims to minimize its own cost function, and each player's cost function depends on actions of the others, then NEs or GNEs are sub-optimal solutions in non-cooperative games.

Motivated by the observations above, in this paper, the online distributed GNE seeking problem in non-cooperative games without full action information is studied. Different from works \cite{F. Salehisadaghiani4}-\cite{ P. Yi5}, players' cost functions are dynamic, and changes can only be seen by players after decisions are made. The offline benchmark for each player is to selfishly minimize its own cost function at each time. Moreover, the shared constraint is modeled as a set of nonlinear convex inequalities, which is more general than those in \cite{S. Liang}-\cite{ P. Yi5}. To address this problem, a novel online distributed algorithm is presented based on consensus algorithms and a primal-dual algorithm. Due to
the existence of nonlinear constraints, the Lagrange multiplier and the argument are coupled in the primal-dual algorithm. This brings challenges to achieving a sublinear regret bound. To overcome this difficulty, a slowly decaying learning rate is employed for ensuring the product of itself and local Lagrange multipliers to be bounded. By implementing the proposed algorithm, each player makes decisions only using its own cost function,
a local set constraint, a local block of the convex inequality constraint, a local estimation on actions of the others, and actions received from its neighbors. We prove that if the graph is connected, then both the regrets and the violation of inequality constraint are bounded by product of a term depending on the deviation of variational GNE sequence and a sublinear function of learning time.

This paper is organized as follows. In Section \ref{se1}, we formulate the problem and present the online distributed GNE seeking algorithm. In Section \ref{se2}, we state our main result and give its proof. In Section \ref{se3}, a simulation example is presented. Section \ref{se4} concludes the whole paper.

{\emph{\textbf{Notations}}}. We use $\mathbb{{N}}^+$ to denote the set of positive integers. For any $T\in\mathbb{{N}}^+$, we denote set $\lfloor T\rfloor=\{1, \cdots, T\}$. $\mathbb{R}^{m}$ and $\mathbb{R}_+^r$ denote $m$-dimensional real vector space and $r$-dimensional non-negative real vector space, respectively. We use $\mathcal{O}(h)$ to denote a general function that is linear with respect to $h$. For a given vector ${x}\in\mathbb{R}^{m}$, $ \Vert{x} \Vert$ denotes the standard Euclidean norm of $x$, i.e., $ \Vert{x} \Vert=\sqrt{{x}^T{x}}$. $\textbf{1}_m$ denotes the $m$-dimensional vector with elements being all ones. For $a\in\mathbb{R}$, we denote $a_+=\max(a, 0)$. For vector $x\in\mathbb{R}^{m}$, we denote $x_+=[x^1_+, \cdots, x^m_+]^T$. Given a differentiable function $f(\cdot): \mathbb{R}^{m}\rightarrow\mathbb{R}$, we use $\nabla_x f(\cdot)$ to denote its gradient.
 For matrices $A$ and $B$, the Kronecker product is denoted by $A\otimes B$. $\lambda_{i}({A})$ represents the $i^{th}$ eigenvalue of square matrix ${A}$. The projection onto a set $\mathcal{K} $ is denoted by $P_\mathcal{K}(\cdot)$.

\section{problem formulation}\label{se1}

\subsection{Online GNE seeking problem}
  A game with $n$ players is denoted by $\Gamma(\mathcal{V}, \chi, J)$. $\mathcal{V}=\{1, \cdots, n\}$ represents the set of players; $\chi=\chi_1\times\cdots\times\chi_n$ denotes the action set of players, where $\chi_i\subset\mathbb{R}^{m}$ is the action set of player $i$; $J=(J_1, \cdots, J_n)$, where $J_i$ is the cost function of player $i$; Let $x=(x_i, x_{-i})$ denote all players' actions, where $x_i$ is the action of player $i$ and $x_{-i}$ denotes actions of players other
than player $i$, i.e., $x_{-i} =[x_1^T, \cdots, x_{i-1}^T, x_{i+1}^T,\cdots ,x_n^T]^T $. For game $\Gamma(\mathcal{V}, \chi, J)$, an action profile ${x^*} = (x_i^{*}, x_{ - i}^{*})$ is called the NE of this game if and only if $J_i(x_i^*, x_{ - i}^*)\leq J_i(x_i, x_{- i}^*)$ holds for any $x_i\in\chi_i$ and $i\in\mathcal{V}$. Moreover, if $\chi_i$ is determined by actions of the other players, then the NE is referred to be a GNE.

Here we consider an online game $\Gamma(\mathcal{V}, \chi, J^t)$, where $J^t=(J_1^t, \cdots, J_n^t)$. We denote the constraint by $\chi:=\chi^s\cap (\Omega_1\times\cdots\times\Omega_n)$, where $\chi^s$ is the shared constraint $\chi^s=\{x\in\mathbb{R}^{nm}|\sum_{i=1}^ng_i(x_i)\leq0\}$, and $g_i(\cdot):\mathbb{R}^m\rightarrow\mathbb{R}^r$ can be nonlinear for any $i\in\mathcal{V}$ and is defined as $g_i=[g_{i1}, \cdots, g_{ir}]^T$, $\Omega_i\subset\mathbb{R}^{m}$ represents player $i$'s private constraint. Then, player $i$'s action set is denoted by $\chi_i(x_{-i})=\{x_i|(x_i, x_{-i})\in\chi\}$. For player $i\in\mathcal{V}$, a set of cost functions are given by $\left\{J_i^1, \cdots, J_i^T\right\}$, where $J_i^t: \mathbb{R}^m\rightarrow\mathbb{R}$, $T\in\mathbb{{N}}^+$ represents the learning time and is unknown to players. At each iteration time $t\in\lfloor T\rfloor$, player $i$ decides an action ${x}_i(t)\in\Omega_i$ under an online algorithm. After the action ${x}_i(t)$ is decided, a local cost function $J_i^t$ is received by player $i$, that is, information associated with cost functions is not available before decisions are made by palyers. In this scenario, for any $t\in\lfloor T\rfloor$, each player intends to solve the optimization problem
 \begin{equation}\label{TMIU}\begin{split}
&\min_{x_i}~~~~~~~J_i^t(x_i, x_{-i})\\
&\textrm{subject}~\textrm{to}~~~{{x_i} \in \chi_i(x_{-i})}.\\
\end{split}
\end{equation}

 Define a pseudo-gradient mapping ${F}^t(x)=[({\nabla _{{x_{1}}}}{J_1}^t(x))^T, $ $\cdots, ({\nabla _{{x_{n}}}}{J_n}^t(x))^T]^T$, some basic assumptions on the cost functions, which are also made in \cite{F. Salehisadaghiani4, K. Lu 6, P. Yi5}, are given as follows.
 \begin{assumption}\label{as1}
For $i\in\mathcal{V}$, $\Omega_i\in\mathbb{R}^{m}$ is a non-empty, compact and convex set; $J_i^t(x_i, x_{-i})$ is differentiable and convex with respect to $x_i$ for any $x_{-i}\in\mathbb{R}^{(n-1)m}$; $g_{ij}(y)$, $j=1, \cdots, r$ are convex and differentiable for any $y\in\mathbb{R}^{m}$ ; The action set $\chi$ is non-empty.
\end{assumption}
\begin{assumption}\label{as2}
(i) (Strong monotonicity) $(F^t(x)-F^t(y))^T(x$ $-y)\geq\mu\|x-y\|^2$, $\forall~x,~y\in\Omega$ for some $\mu>0$; \\
(ii) (Lipschitz continuity) $ \Vert\nabla_{x_{i}}J_i^t(x_{i}, w)-\nabla_{x_{i}}J_i^t(x_{i}, z) \Vert\leq \ell \Vert w-z \Vert$ for some $\ell>0$, $\forall~x_i\in \Omega_{i};~ w, z\in\mathbb{R}^{m}$, $i\in \mathcal{V}$.
 \end{assumption}

Under Assumption \ref{as1}, we know that $\Omega_i$ is convex and compact. Then, for any $i\in{\mathcal{V}}$ and $x_i\in\Omega_i$, $\|x_i\|$, $\|g_i(x_i)\|$, $\|\nabla_{x_i}J_i^t(x)\|$ and $\|\nabla_{x_i}g_i(x_i)\|$ are bounded, and we denote
\begin{equation}\label{eq4jjk000}\left\{ {\begin{array}{*{20}{c}}
  \begin{split}
  & \kappa_1=\sup_{x\in\Omega_i}\|x_i\|,~\kappa_2=\sup_{x\in\Omega_i}\|g_i(x_i)\|,\\
  &\kappa_3=\sup_{x\in\Omega}\|\nabla_{x_i}J_i^t(x)\|,~\kappa_0=\sup_{x_i\in\Omega_i}\|\nabla_{x_i}g_i(x_i)\|
\end{split}\end{array}} \right. \forall i\in\mathcal{V}.\end{equation}
Note that set $\chi$ is convex and compact. Together with the convexity and differentiability of cost functions, by Theorem
3.9 in \cite{F. Facchinei2}, we know that for any $t\in\lfloor T\rfloor$, every solution $x^*(t)\in\chi$ to the following variational inequality is a GNE of game $\Gamma(\mathcal{V}, \chi, J^t)$:
\begin{equation}\label{eq16***}\begin{split}
~~\left(F^t(x^*(t))\right)^T(x-x^*(t))\geq0 ~\emph{\emph{for}}~ \emph{\emph{all}}~ x\in\chi.
\end{split}\end{equation}
The solution $x^*(t)\in\chi$ in (\ref{eq16***}) is called a variational GNE. Moreover, by the strong monotonicity condition in Assumption \ref{as2}, it follows from Theorem 2. 3. 3 in \cite{F. Facchinei55} that variational inequality (\ref{eq16***}) has a unique solution. Accordingly, Assumptions \ref{as1} and \ref{as2} ensure existence and uniqueness of the variational GNE. Currently, it is rather difficult to seek all GNEs even if the game is offline. Since the variational GNE has the economic
interpretation of no price discrimination and enjoys good stability, seeking the variational GNE is the main goal in the study of games with shared constraints \cite{S. Liang}-\cite{ P. Yi5}.
 Given some $\gamma(t)>0$, based on variational inequality (\ref{eq16***}), define a Lagrange function $L^t(x, y)=\gamma(t)\Big(\sum_{i=1}^n J_i^t(x_i, x_{-i}^*(t))+\gamma(t)y^T\Big(\sum_{j=1, j\neq i}^ng_j(x_j^*(t))$ $+g_i(x_i)\Big)\Big)$, where $x\in\Omega$, $y\in\mathbb{R}_+^m$, and ${x^*(t)} = (x_i^{*}(t), x_{ - i}^{*}(t))$ is the variational GNE of game $\Gamma(\mathcal{V}, \chi, J^t)$. Using primal-dual theory \cite{P. Yi115, K. Lu 6, F. Facchinei2} and by the fact that $ {x_i^*(t)} =P_{\Omega_i}[x_i^*(t)]$, we know that for any $\gamma(t)\geq0$, there exists a bounded Lagrange multiplier $y^*(t)\in \mathbb{R}_+^m$ such that the following Karush-Kuhn-Tucker condition (KKT condition)
\begin{equation}\label{eq4jjk}\left\{ {\begin{array}{*{20}{c}}
  \begin{split}
  & {x_i^*(t)} = {P_{{\Omega_i}}}[{x_i^*(t)} - \gamma(t)(\nabla_{x_i}J_i^t(x_i^{*}(t), x_{ - i}^{*}(t)) + \gamma(t)\nabla_{x_i}g_i(x_i^*(t))y^*(t))] \\
   &{{y^*(t)} = \left[{y^*(t)} + \sum_{i=1}^ng_i(x_i^*(t))\right]_+} \\
\end{split}\end{array}} \right. \forall i\in\mathcal{V}\end{equation}
 where $\nabla_{x_i}g_i(\cdot)=[\nabla_{x_i}g_{i1}(\cdot), \cdots, $ $\nabla_{x_i}g_{ir}(\cdot)]$.

 Any learning or online algorithm should mimic the performance of its offline counterpart, and the gap between them is {regret}. Here the offline benchmark for each player is to minimize its own cost function at each time. By the definition of GNEs, we know that $(x_i^*(t), x_{-i}^*(t))$ is a GNE of (\ref{TMIU}) if and only if $x_i^*(t)$ is the solution to the following optimization
  \begin{equation}\label{TMIU12}\begin{split}
&\min_{x_i}~~~~~~~J_i^t(x_i, x_{-i}^*(t))\\
&\textrm{subject}~\textrm{to}~~~{{x_i} \in \chi_i(x_{-i}^*(t)})\\
\end{split}
\end{equation}
for any $i\in\mathcal{V}$. Based on (\ref{TMIU12}) and motivated by \cite{E. C. Hall}, the regret is defined as follows
\begin{equation}\label{A}\begin{split}
\mathcal{R}_i(T)=\sum_{t=1}^T\Big(J_i^t(x_i(t), x_{-i}^*(t))-J_i^t(x^*(t))\Big),~i\in\mathcal{V}
\end{split}\end{equation}
where ${x^*}(t) = (x_i^{*}(t), x_{ - i}^{*}(t))$ is the variational GNE satisfying (\ref{eq4jjk}) at iteration time $t$. Accordingly, the violation of the inequality constraint is defined as
\begin{equation}\label{EMM}\begin{split}
\mathcal{R}_g(T)=\left\|\left[\sum_{t=1}^T\sum_{i=1}^ng_i(x_i(t))\right]_+\right\|.
\end{split}\end{equation}

An online algorithm performs well if both (\ref{A}) and (\ref{EMM}) sublinearly increase, i.e., $\lim_{T\rightarrow\infty}\mathcal{R}_i(T)/T$ $=0$ and $\lim_{T\rightarrow\infty}\mathcal{R}_g(T)/T$ $=0$. However, if the variational GNE sequence $\{{x}^*(t)\}_{t=0}^T$ fluctuate drastically, it could be impossible to keep dynamic regrets sublinear. Motivated by \cite{E. C. Hall}-\cite{A. Jadbabaie8}, we use the following deviation of the GNE sequence $\{{x}^*(t)\}_{t=0}^T$ to describe the difficulty:
\begin{equation}\label{QQqq}
\Theta_T=\sum_{t=0}^T\|{x}^*(t+1)-{x}^*(t)\|.
\end{equation}
\begin{problem} In game $\Gamma(\mathcal{V}, \chi, J^t)$, $J_i$ may be depended on actions of players who are not player $i$'s neighbors. Suppose that player $i$ can only communicate with its neighbors via
communication graph  $\mathcal{G}(A)$, and has access to the information associated with
$J_i$, $\Omega_i$, $g_i$ for any $i\in\mathcal{V}$. The goal of this paper is to design an online distributed algorithm
 for the players to seek variational GNEs of the game $\Gamma(\mathcal{V}, \chi, J^t)$, and performance of the algorithm is measured by using regret (\ref{A}) and violation (\ref{EMM}).
\end{problem}
\begin{remark}
If each cost function $J^t_i$ is fixed to be $J_i$, and $g_i(x_i)=[(A_ix_i-b_i)^T, -(A_ix_i-b_i)^T]^T$ for some matrix $A_i$ and vector $b_i$ with suitable dimension, then Problem 1 is reduced to be an offline GNE seeking problem studied in \cite{S. Liang,  P. Yi5}. Different from them, we investigate the case where cost functions are dynamic and the constraints are nonlinear.
\end{remark}

\subsection{An online distributed algorithm for GNEs}
Before presenting our algorithm, we denote the communication graph by an undirected graph $\mathcal{G}({A})$, where  $A=(a_{ij})_{n\times n}$ represents the weighted matrix. In $\mathcal{G}({A})$, the set of player $i$'s neighbors is denoted by $\mathcal{N}_i$. If $j\in\mathcal{N}_i$, then $a_{ij}=a_{ji}>0$; Otherwise, $a_{ij}=a_{ji}=0$. Moreover, the following connectivity assumption associated with graph $\mathcal{G}({A})$ is presented.
\begin{assumption}\label{as3}
$\mathcal{G}({A})$ is connected. Moreover, $0<a_{ii}<1$ for any $i\in \mathcal{V}$ and $A\textbf{1}_n=\textbf{1}_n$.
 \end{assumption}

We use $A_i^-$ to denote a submatrix that is formed by removing the $i_{th}$ row and the $i_{th}$ column of weighted matrix ${A}$. Define matrix $\Lambda_i=\diag(a_{1i}, \cdots, a_{(i-1)i}, a_{(i+1)i}, \cdots, a_{(n)i})$, it is obvious that $A_i^-=(A_i^-)^T$ and $((I_{n-1}-A_i^-)-\Lambda_i)\textbf{1}_{n-1}=0$.
Under Assumption \ref{as3}, we know that there exists a path from node $i$ to any other one. Based on Lemma 3 in \cite{Hong4}, we know that $I_{n-1}-A_i^-$ is positive definite, which implies that $\lambda_k(A_i^-)<1$. Using Gerschgorin's disk theorem and the fact that $0<a_{ii}<1$, it is not difficult to verify that $\lambda_k(A_i^-)>-1$ Thus, $-1<\lambda_k(A_i^-)<1$ for any $k=1, \cdots, n-1$. Throughout this paper, we denote
\begin{equation}\label{oo}\begin{split}
\lambda=\max_{1\leq i\leq n, 1\leq k\leq n-1}\left|\lambda_k(A_i^-)\right|.
\end{split}\end{equation}
It is obvious that $0<\lambda<1$. Furthermore, for symmetric and stochastic matrix ${A}$, its singular values can be sorted in a non-increasing fashion
\begin{equation}\label{uuu}\begin{split}
1=\sigma_1(A)\geq\sigma_2(A)\geq\cdots\geq\sigma_n(A)\geq0.
\end{split}\end{equation}
Under Assumption \ref{as3}, one knows that $0<\sigma_2(A)<1$ \cite{R. A. Horn}.

 Let vector $\textbf{x}_{-i}=[{x}_{i1}^T, \cdots, $  ${x}_{i(i-1)}^T, {x}_{i(i+1)}^T,$  $ \cdots, {x}_{in}^T]^T$ denote player $i$'s estimates on all the players' actions but its own's, where ${x}_{ij}$ is the player $i$'s estimate on player $j$'s action. For ease, we denote $\textbf{x}_{i}=[{x}_{i1}^T, \cdots, x_{in}^T]^T$, where ${x}_{ii}={x}_i$ denote player $i$'s real action. To solve Problem 1, we propose the following algorithm for player $i$, $i\in \mathcal{V}$:
 \begin{equation}\label{eq4}
 \left\{
\begin{split}
&{{{x}}}_{ih}(t+1)=\sum_{k \in {\mathcal{N}_i/\{h\}}} a_{ik}{x}_{kh}(t)+a_{ih}x_h(t),~h\neq i\\
&{x}_i(t+1)=(1-\gamma(t))x_i(t)+\gamma(t)P_{\Omega_i}\Big[ {x}_i(t)-\gamma(t)(\nabla_{x_i}J_i^t(\textbf{x}_i(t))+\gamma(t)\nabla_{x_i}g_i(x_i(t))y_i(t))\Big]\\
& {y_i(t+1)} =\Big[(1-\gamma^2(t))\sum_{j\in {\mathcal{N}_i}}a_{ij}{y}_j(t)+\gamma(t)g_i(x_i(t))\Big]_+\\
\end{split}
\right.
\end{equation}
 where $y_i\in\mathbb{R}^r$ is player $i$'s estimate on the Lagrange multiplier, $\gamma(t)$ is a non-increasing learning rate such that $0\leq\gamma(t)\leq1$, and the initial states are chosen as ${x}_i(1)\in{\Omega}_i$,
$y_i(1)=0$, $x_{ih}(1)=0$ for any $i\neq h$. In algorithm (\ref{eq4}), each player updates estimates on actions of others by a leader-following consensus algorithm \cite{Hong4}; Player $i$ updates $x_i(t)$ and $y_i(t)$ by using a distributed primal-dual strategy, which is motivated by the consensus algorithm \cite{Guan Yongqiang1, A. Li8} and the primal-dual strategy \cite{Nesterov}. Note that players make decisions only using local state information and their own cost functions in the past time, thus, algorithm (\ref{eq4}) is online and distributed.

\section{Main results}\label{se2}
In this section, we state our main result and give its proof in
detail.
\begin{theorem}  Under Assumptions 1-\ref{as3}, by algorithm (\ref{eq4}), regrets (\ref{A}) and violation (\ref{EMM}) are bounded by
\begin{equation}\label{eq58}\begin{split}
&\mathcal{R}_i(T)\leq \mathcal{O}\left(\sqrt{T\left(\frac{\Theta_T+1}{\gamma^2(T)}+\sum_{t=1}^{T}\gamma(t)\right)}\right),~i\in \mathcal{V}
\end{split}\end{equation}
and
\begin{equation}\label{eq70}\begin{split}
\mathcal{R}_g(T)\leq  \mathcal{O}\left(\frac{\sqrt{\left(\frac{\Theta_T+1}{\gamma^2(T)}+\sum_{t=1}^{T}\gamma(t)\right)\left(1+\sum_{t=1}^{T}\gamma^2(t)\right)}}{\gamma(T)}\right)
\end{split}\end{equation}
where $\Theta_T$ is defined in (\ref{QQqq}).
\end{theorem}

Results in Theorem 1 indicate that the sublinearity of bounds in
 (\ref{eq58}) and (\ref{eq70}) is determined by $\Theta_T$ and $\gamma(t)$. If we set the learning rate to be fixed by selecting $\gamma(t)=C$ for some $C>0$, then both of bounds in (\ref{eq58}) and (\ref{eq70}) are equivalent to $\mathcal{O}\left(\sqrt{T(T+\Theta_T)}\right)$, which do not increase sublinearly. To keep sublinearity of bounds in (\ref{eq58}) and (\ref{eq70}),  diminishing learning rate is necessary. Let diminishing learning rate be $\gamma(t)=C^{\eta}({Dt+C})^{-\eta}$ for some $C,~D>0$ and $0<\eta<\frac{1}{2}$, we have

\[\begin{split}
 \sum_{t=1}^T\gamma(t)&=\sum_{t=1}^{T}C^{\eta}({Dt+C})^{-\eta}\\
 &\leq1+\int_{0}^{T}C^{\eta}({Dt+C})^{-\eta}dt\\
& \leq \mathcal{O}\left(T^{1-\eta}\right)
 \end{split}\]
and
$$\sum_{t=1}^T\gamma^2(t)=\sum_{t=1}^{T}C^{-2\eta}({Dt+C})^{-2\eta}\leq \mathcal{O}\left(T^{1-2\eta}\right).$$
Then, based on Theorem 1, we have the following corollary.

\begin{corollary}  Under Assumptions 1-\ref{as3}, if the learning rate is given as $\gamma(t)=C^{\eta}({Dt+C})^{-\eta}$ for some $C,~D>0$ and $0<\eta<\frac{1}{2}$, then by algorithm (\ref{eq4}), regrets (\ref{A}) and violation (\ref{EMM}) are bounded by
\begin{equation}\label{eq58c}\begin{split}
&\mathcal{R}_i(T)\leq \mathcal{O}\left(T^{\frac{1}{2}+\eta}\left(\sqrt{1+\Theta_T}+T^{\frac{1-3\eta}{2}}\right)\right),~i\in \mathcal{V}
\end{split}\end{equation}
and
\begin{equation}\label{eq70c}\begin{split}
\mathcal{R}_g(T)\leq  \mathcal{O}\left(T^{\frac{1}{2}+\eta}\left(\sqrt{1+\Theta_T}+T^{\frac{1-3\eta}{2}}\right)\right)
\end{split}\end{equation}
where $\Theta_T$ is defined in (\ref{QQqq}).
\end{corollary}

In Corollary 1, it is obvious that terms $T^{\frac{1}{2}+\eta}$ and $T^{\frac{1}{2}+\eta}\cdot T^{\frac{1-3\eta}{2}}=T^{1-\frac{\eta}{2}}$ are sublinear with $T$ due to the fact that $0<\eta<\frac{1}{2}$, then $\Theta_T$ is a significant factor that influences the sublinearity of the bounds. Note that if the increasing rate of $\Theta_T$ is sublinear with ${T}^{1-2\eta}$, i.e., $\lim_{T\rightarrow\infty}\frac{\Theta_T}{{T^{1-2\eta}}}=0$, then $\lim_{T\rightarrow\infty}\frac{T^{\frac{1}{2}+\eta}\sqrt{1+\Theta_T}}{T}=0$, which implies that both $\mathcal{R}_i(T)$ and $\mathcal{R}_g(T)$ sublinearly increase with $T$. In this case, online distributed algorithm (\ref{eq4}) performs well. Particularly, let $\eta=\frac{1}{4}$, if the increasing rate of $\Theta_T$ is sublinear with ${\sqrt{T}}$, then both $\mathcal{R}_i(T)$ and $\mathcal{R}_g(T)$ sublinearly grow with $T$. If
variational GNE sequence $\{{x}^*(t)\}_{t=1}^T$ fluctuates drastically, $\Theta_T$ could become linear with ${T}^\eta$, then the bounds in Theorem 1 cannot keep $\mathcal{R}_i(T)$ and $\mathcal{R}_g(T)$ sublinear. This is natural since even in the optimization problems having globally optimal solutions \cite{E. C. Hall}-\cite{A. Jadbabaie8}, the problem is insolvable in worst cases. Furthermore, if cost functions are time-invariant, and let $J_i^t=J_i$ and $x^*(t)=(x_i^*, x_{-i}^*)$ for any $i\in \mathcal{V}$, then $\Theta_T=0$, and the game $\Gamma(\mathcal{V}, \chi, J^t)$ is reduced to be an offline case. By (\ref{eq58c}) and definitions of regrets in (\ref{A}), there holds $\lim_{T\rightarrow\infty}\frac{\sum_{t=1}^TJ_i(x_i(t), x_{-i}^*)}{T}-J_i(x^*)=0$. Replacing players' actions with the average value defined $\bar{x}_i(t)=\frac{\sum_{k=1}^tx_i(k)}{t}$, and using convexity of $J_i$ yield $\lim_{t\rightarrow\infty}J_i(\bar{x}_i(t), x_{-i}^*)-J_i(x^*)=0$, which implies that $\lim_{t\rightarrow\infty}\bar{x}_i(t)=x_i^*$. Accordingly, $(\bar{x}_i(t), \bar{x}_{-i}(t))$ asymptotically converges to the variational GNE. Thus, the methods and results are also applicable to the offline cases studied in \cite{F. Salehisadaghiani4, F. Salehisad2, S. Liang}-\cite{K. Lu 6}.

In what follows, some lemmas are provided to prove Theorem 1. First, the boundedness of $\|\gamma(t)y_i(t)\|$ is presented.
\begin{lemma}\label{le34}
Under Assumption \ref{as1}, $\|\gamma(t)y_i(t)\|\leq{\sqrt{n}\kappa_2}$ for any $2\leq t\leq T$ and $i\in\mathcal{V}$, where $\kappa_2$ is defined in (\ref{eq4jjk000}).
 \end{lemma}
\begin{proof}
See APPENDIX. A.
\end{proof}

Now we present the bound of the error between estimate on each player's action and its real action, as well as a bound of the error between each player's estimate on the Lagrange multiplier
and their average. Before going on, we denote the error by $e_{ih}(t)=x_{ih}(t)-x_h(t)$ and $e_i=[e_{1i}^T, \cdots, e_{(i-1)i}^T, e_{(i+1)i}^T, $  $\cdots, e_{ni}^T]^T$.
\begin{lemma}\label{le5}
 Under Assumptions \ref{as1} and \ref{as3}, for any $i\in\mathcal{V}$ and $2\leq t\leq T$,\\
(i) $\|e_{i}(t)\|\leq \lambda^{t-1}\|{e}_{i}(1)\|+2\sqrt{n-1}\kappa_1\sum_{k=0}^{t-2}\lambda^{k}\gamma(t-k-1)$;\\
(ii) $\|e_{i}(t)\|^2\leq \lambda^{t-1}\|{e}_{i}(1)\|^2+\rho\sum_{k=0}^{t-2}\lambda^{k}\gamma(t-k-1)$;\\
(iii) $\|{y_i(t)}-{\bar{y}(t)} \|\leq2(n+{\sqrt{n}})\kappa_2\sum_{k=0}^{t-2}(\sigma_2(A))^k\gamma(t-1-k)$;\\
(iv) $\|{y_i(t)}-{\bar{y}(t)}\|^2\leq\varrho\sum_{k=0}^{t-2}(\sigma_2(A))^{k}\gamma(t-k-1)$\\
where ${\bar{y}(t)}=\frac{1}{n}\sum_{i=1}^n{y}_i(t)$, $\rho=\frac{8({n-1})\kappa_1^2}{1-\lambda}$, $\varrho=\frac{4(n+{\sqrt{n}})^2\kappa_2^2}{1-\sigma_2(A)}$, $\kappa_1$ and $\kappa_2$ are defined in (\ref{eq4jjk000}), and $\lambda$ and $\sigma_2(A)$ are defined in (\ref{oo}) and (\ref{uuu}), respectively.
  \end{lemma}
 \begin{proof} See APPENDIX. B.
\end{proof}

Note that $\lim_{t\rightarrow\infty}\lambda^{t-1}=0$. If $\gamma(t)$ is diminishing, i.e., $\lim_{t\rightarrow\infty}\gamma{(t)}=0$, then $\lim_{t\rightarrow\infty}\sum_{k=0}^{t-2}\lambda^{k}\gamma(t-k-1)$ and $\lim_{t\rightarrow\infty}\sum_{k=0}^{t-2}(\sigma_2(A))^{k}\gamma(t-k-1)=0$, accordingly, $\lim_{t\rightarrow\infty}\|e_{i}(t)\|=0$ and $\lim_{t\rightarrow\infty}\|{y}_i(t)-\bar{y}(t)\|=0$, which implies that each player can estimate real actions of others, and all ${y}_i(t)$, $i=1, \cdots, n$ will approach to a common value as time evolves.
Next, we establish an upper bound of the accumulated square error between players' actions and the GNE.
\begin{lemma}\label{le2}Under Assumptions \ref{as2} and \ref{as1}, let $x(t)=(x_i(t), x_{-i}(t))$,

\begin{equation}\label{eq9}\begin{split}
&\sum_{t=1}^T\|x(t)-x^*(t)\|^2\\
&\leq\frac{2\kappa_1\sqrt{n}}{\mu\gamma^2(T)}(\Theta_T+{\kappa_1\sqrt{n}})+\frac{(1+\kappa_0)\ell^2}{2\mu}\sum_{t=1}^T\sum_{i=1}^n\|e_i(t)\|^2\\
&~~~+\frac{\pi_1}{\mu}\sum_{t=1}^T\gamma(t)+\frac{\pi_2}{\mu}\sum_{t=1}^T\sum_{i=1}^n\|e_i(t)\|+\frac{\kappa_2}{\mu}\sum_{t=1}^T\sum_{i=1}^n\|y_i(t)\\
&~~~-\bar{y}(t)\|-\frac{1}{\mu}\sum_{t=1}^T\sum_{i=1}^n \gamma(t)(g_i(x_i(t)))^Ty_i(t)
\end{split}\end{equation}
where $\pi_1=\tau/2$ $+n^2(\kappa_0^2+\kappa_0)\kappa_2^2+2\kappa_0\kappa_3\kappa_2n\sqrt{n}+n\kappa_0\theta^2$, $\tau=4n\kappa_0\theta(\theta\kappa_0+\kappa_1+\kappa_3)+4n\kappa_3^2$, $\pi_2=2\sqrt{n-1}\ell(\kappa_1+\kappa_3)$, $\theta=\sup_{ t\in\lfloor T\rfloor}\|y^*(t)\|$, $\kappa_i$, $i=0, \cdots, 3$ are defined in (\ref{eq4jjk000}), and $\mu$ is the strong monotonicity parameter given in Assumption 1.
 \end{lemma}
\begin{proof}
See APPENDIX. C.
 \end{proof}

Then, the lower bound of $\sum_{t=1}^T\sum_{i=1}^n \gamma(t)(g_i(x_i(t)))^Ty_i(t)$ in (\ref{eq9}) is presented.
\begin{lemma}\label{le3} Under Assumption \ref{as1}, for any $y\in\mathbb{R}^r_+$,
\begin{equation}\label{eq14}\begin{split}
&\sum_{t=1}^T\sum_{i=1}^n \gamma(t)\big((g_i(x_i(t)))^T{y}_i(t)-(g_i(x_i(t)))^T{y}\big)\\
&\geq-\frac{n}{2}\left(1+\sum_{t=1}^T\gamma^2(t)\right)\|y\|^2-\frac{9}{2}\sum_{t=1}^T\sum_{i=1}^n\| {y}_i(t)-\bar{y}(t)\|^2\\
&~~~-n\kappa_2^2(1+n)\sum_{t=1}^T{\gamma(t)}-2\kappa_2(1+\sqrt{n})\sum_{t=1}^T\sum_{i=1}^n\|{y}_i(t)-{\bar{y}}(t)\|.\\
\end{split}\end{equation}
 \end{lemma}
\begin{proof}
See APPENDIX. D.
 \end{proof}

Now, we can present the proof of Theorem 1.

\textbf{Proof of Theorem 1.}
Letting $\gamma(t)=\frac{c}{\sqrt[3]{dt+c}}$, we have
 \[\begin{split}
 &\sum_{t=2}^T\sum_{k=0}^{t-2}\lambda^{k}\gamma(t-k-1)=\sum_{k=0}^{T-2}\lambda^{k}\sum_{t=1}^{T-k-1}\gamma(t)\leq \mathcal{O}\left(\sum_{t=1}^{T}\gamma(t)\right)
 \end{split}\]
where the equation holds by changing the order of summations. Similarly, $\sum_{t=2}^T\sum_{k=0}^{t-2}(\sigma(A))^{k}\gamma(t-k-1)\leq\mathcal{O}\left(\sum_{t=1}^{T}\gamma(t)\right)$. By in Lemma \ref{le5}, we have
\begin{equation}\label{eq51}\left\{ {\begin{array}{*{20}{c}}
  \begin{split}
 &\sum_{t=1}^T\sum_{i=1}^n\|e_i(t)\|\leq \mathcal{O}\left(\sum_{t=1}^{T}\gamma(t)\right) \\
   &\sum_{t=1}^T\sum_{i=1}^n\|e_i(t)\|^2\leq  \mathcal{O}\left(\sum_{t=1}^{T}\gamma(t)\right) \\
   &\sum_{t=1}^T\sum_{i=1}^n\|y_i(t)-\bar{y}(t)\|\leq \mathcal{O}\left(\sum_{t=1}^{T}\gamma(t)\right)\\
   &\sum_{t=1}^T\sum_{i=1}^n\|y_i(t)-\bar{y}(t)\|^2\leq\mathcal{O}\left(\sum_{t=1}^{T}\gamma(t)\right).
\end{split}\end{array}} \right.\end{equation}
From Lemma \ref{le34}, we know that $\|\gamma(t)y_i(t)\|\leq{\sqrt{n}\kappa_2}$. Submitting (\ref{eq51}) into (\ref{eq9}) in Lemma \ref{le2} yields
\begin{equation}\label{eq55}\begin{split}
&\sum_{t=1}^T\|x(t)-x^*(t)\|^2+\frac{1}{\mu}\sum_{t=1}^T\sum_{i=1}^n \gamma(t)(g_i(x_i(t)))^Ty_i(t)\leq\mathcal{O}\left(\frac{\Theta_T+1}{\gamma^2(T)}+\sum_{t=1}^{T}\gamma(t)\right).
\end{split}\end{equation}
Furthermore, submitting (\ref{eq51}) into (\ref{eq14}) in Lemma \ref{le3} yields
\begin{equation}\label{eq56}\begin{split}
&-\sum_{t=1}^T\gamma(t)\sum_{i=1}^n (g_i(x_i(t)))^Ty_i(t)+\sum_{t=1}^T\gamma(t)\sum_{i=1}^n\big(g_i(x_i(t))\big)^T{y}-\frac{n}{2}\left(1+\sum_{t=1}^T\gamma^2(t)\right)\|y\|^2\leq\mathcal{O}\left(\sum_{t=1}^{T}\gamma(t)\right)\\
\end{split}\end{equation}
for any $y\in\mathbb{R}^r_+$. Combining (\ref{eq55}) and (\ref{eq56}) results in that
\begin{equation}\label{eq57}\begin{split}
&\sum_{t=1}^T\|x(t)-x^*(t)\|^2+\frac{1}{\mu}\sum_{t=1}^T\gamma(t)\sum_{i=1}^n(g_i(x_i(t)))^T{y}\\
&-\frac{n}{2\mu}\left(1+\sum_{t=1}^T\gamma^2(t)\right)\|y\|^2\leq \mathcal{O}\left(\frac{\Theta_T+1}{\gamma^2(T)}+\sum_{t=1}^{T}\gamma(t)\right).
\end{split}\end{equation}
Note that $\|J_i^t(x_i(t), x_{-i}^*(t))-J_i^t(x^*(t))\|\leq\kappa_3\|x_i(t)-x_i^*(t)\|$. Thus,
\begin{equation}\label{eqgksq}\begin{split}
\mathcal{R}_i(T)&\leq{\kappa_3}\sum_{t=1}^T\|x_i(t)-x_i^*(t)\|\\
&\leq {\kappa_3}\sqrt{T\sum_{t=1}^T\|x(t)-x^*(t)\|^2}.
\end{split}\end{equation}
Letting $y$ in inequality (\ref{eq57}) be 0, we have $\sum_{t=1}^T\|x(t)-x^*(t)\|^2\leq \mathcal{O}\left(\frac{\Theta_T+1}{\gamma^2(T)}+\sum_{t=1}^{T}\gamma(t)\right)$. Together with (\ref{eqgksq}), it immediately implies inequality (\ref{eq58}).

Furthermore, due to the arbitrariness of $y\in\mathbb{R}^m_+$, letting $$y= \frac{\left[\sum_{t=1}^T\gamma(t)\sum_{i=1}^ng_i(x_i(t))\right]_+}{n\left(1+\sum_{t=1}^T\gamma^2(t)\right)}$$ one has
\[\begin{split}
&\left(\sum_{t=1}^T\gamma(t)\sum_{i=1}^ng_i(x_i(t))\right)^T{y}-\frac{n}{2}\left(1+\sum_{t=1}^T\gamma^2(t)\right)\|y\|^2\\
&=\frac{\left[\sum_{t=1}^T\gamma(t)\sum_{i=1}^ng_i(x_i(t))\right]_+^2}{2n\left(1+\sum_{t=1}^T\gamma^2(t)\right)}\\
&\geq\frac{(\gamma(T)\mathcal{R}_g(t))^2}{2n\left(1+\sum_{t=1}^T\gamma^2(t)\right)}.
\end{split}\]
Together with (\ref{eq57}), it follows
\begin{equation}\label{eq59}\begin{split}
&\gamma^2(T)(\mathcal{R}_g(t))^2\leq \mathcal{O}\left({{\left(\frac{\Theta_T}{\gamma^2(T)}+\sum_{t=1}^{T}\gamma(t)\right)\left(1+\sum_{t=1}^{T}\gamma^2(t)\right)}}\right).
\end{split}\end{equation}
Inequality (\ref{eq59}) implies inequality (\ref{eq70}). This completes the proof.
\QEDA

\section{A simulation example}\label{se3}
In this section, we illustrate the achieved results by using
our algorithm to deal with an online Nash-Cournot game $\Gamma(\mathcal{V}, \chi, J^t)$ with production constraints and market capacity constraints. In the Nash-Cournot game, there exist five firms (players) that produce same production, and these firms are labeled by index set $\mathcal{V}=\{1,\cdots, 5\}$. The firms communicate with each other via a connected graph $\mathcal{G}(A)$ shown in Fig. \ref{fig1}, where each element of weighted matrix $A=(a_{ij})_{n\times n}$ is set
to be $a_{ij}=\frac{1}{|\mathcal{N}_i|}$, and $|\mathcal{N}_i|$ is the element number in $\mathcal{N}_i$. Let the quantity produced by firm $i$ be ${x}_i\in\mathbb{R}$. Due to existence of some changeable factors such as marginal costs, the production cost and the demand price could be time-varying \cite{yehu55}. The production cost and the demand price of firm $i$ are given by $p_i^t(x_i)=\alpha_i(t)x_i$ and $d_i^t(x)=\beta_i(t)-\sum_{j=1}^5x_j$, respectively, where $\alpha_i(t), \beta_i(t)>0$ and $t\in[0, T]$. Then, based on formulations in \cite{yehu55} and \cite{Gadjov11}, the overall cost function of firm $i$ follows that $J_i^t(x_i, x_{-i})=p_i^t(x_i)-x_id_i^t(x)$ for any $t\in[0, T]$ and $i\in\mathcal{V}$. Moreover, the production constraint of each firm is given by $\Omega_i\subset \mathbb{R}$, and the market capacity constraint is given by shared inequality constraint $\sum_{i=1}^5x_i\leq \sum_{i=1}^5l_i$.
To achieve maximum benefit, each firm aims to minimize its own cost function $J_i^t(x_i, x_{-i})$.

In this formulation, we set $T=300$, $\alpha_i(t)=\sin(t/12)$, $\beta_i(t)=45+5i-0.5i\sin(t/12)$, $l_1=10$, $l_2=l5=15$, $l_3=8$, $l_4=8$ and $\Omega_i=\{x|0\leq x\leq30\}$, $i\in\mathcal{V}$. In the offline setting, we denote the GNE of game
$\Gamma(\mathcal{V}, \chi, J^t)$ by $x^*(t)=[x_1^*(t), \cdots, x_n^*(t)]^T$. At each iteration time, $x_i^*(t)$ approximates to $P_{\Omega_i}(\xi_i)$, where $\xi_i=\frac{8\sin(t/12)}{3}-5+5i$. Now suppose that player $i$ can only have access to $J_i^{t-1}, l_i, \Omega_i$ and its neighbors' actions. Algorithm (\ref{eq4}) is applied to the problem. Initial values are given by $x_1(0)=0$, $x_2(0)=x_5(0)=30$, $x_3(0)=x_4(0)=10$, $x_{ij}(0)=10$ and $y_i(0)=1$, $i,~j\in\mathcal{V}$, and the learning rate is set to be that $\gamma(t)={{\sqrt[3]{\frac{6}{0.1t+6}}}}$.
Fig. \ref{fig2} shows the trajectories of players' real actions. The trajectories of average regrets ${\mathcal{R}}_i(t)/t$, $i\in\mathcal{V}$ and average violation ${\mathcal{R}}_g(t)/t$ are shown in Fig. \ref{fig3} and in Fig. \ref{fig4}, respectively. From Fig. \ref{fig3} and  Fig. \ref{fig4}, we see that both ${\mathcal{R}}_i(t)/t$ and ${\mathcal{R}}_g(t)/t$ approximately decay to zero after a period of time, which implies that both the regret and the violation of inequality constraint sublinearly increase. These observations are consistent with the results
established in Theorem 1.
\begin{figure}
\begin{minipage}[t]{0.5\linewidth}
\centering
\includegraphics[width=1\textwidth]{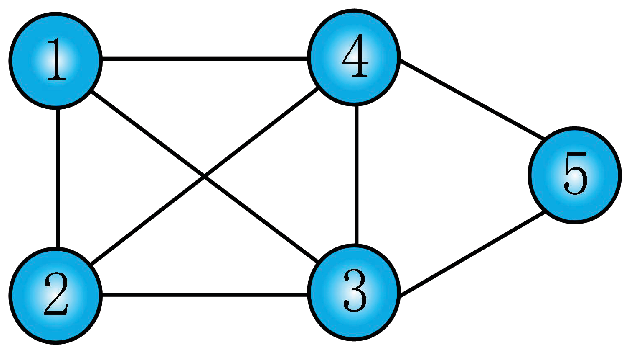}
\caption{Communication graph $\mathcal{G}(A)$.}\label{fig1}
\end{minipage}
\begin{minipage}[t]{0.5\linewidth}
\centering
\includegraphics[width=0.9\textwidth]{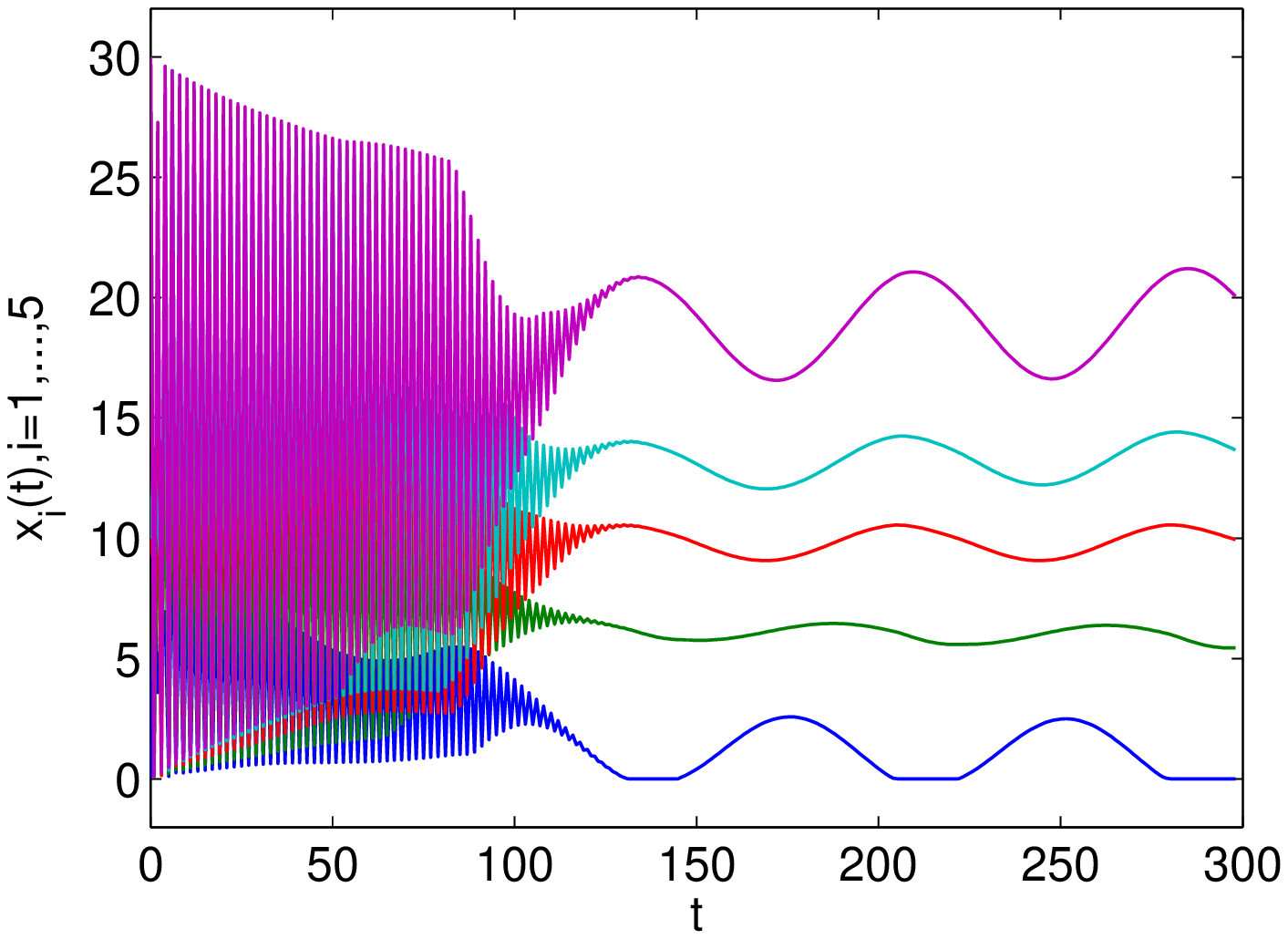}
\caption{The trajectories of players' actions ${x}_i(t)$, $i=1,\cdots,5$.}\label{fig2}
\end{minipage}
\end{figure}

\begin{figure}
\begin{minipage}[t]{0.5\linewidth}
\centering
\includegraphics[width=0.9\textwidth]{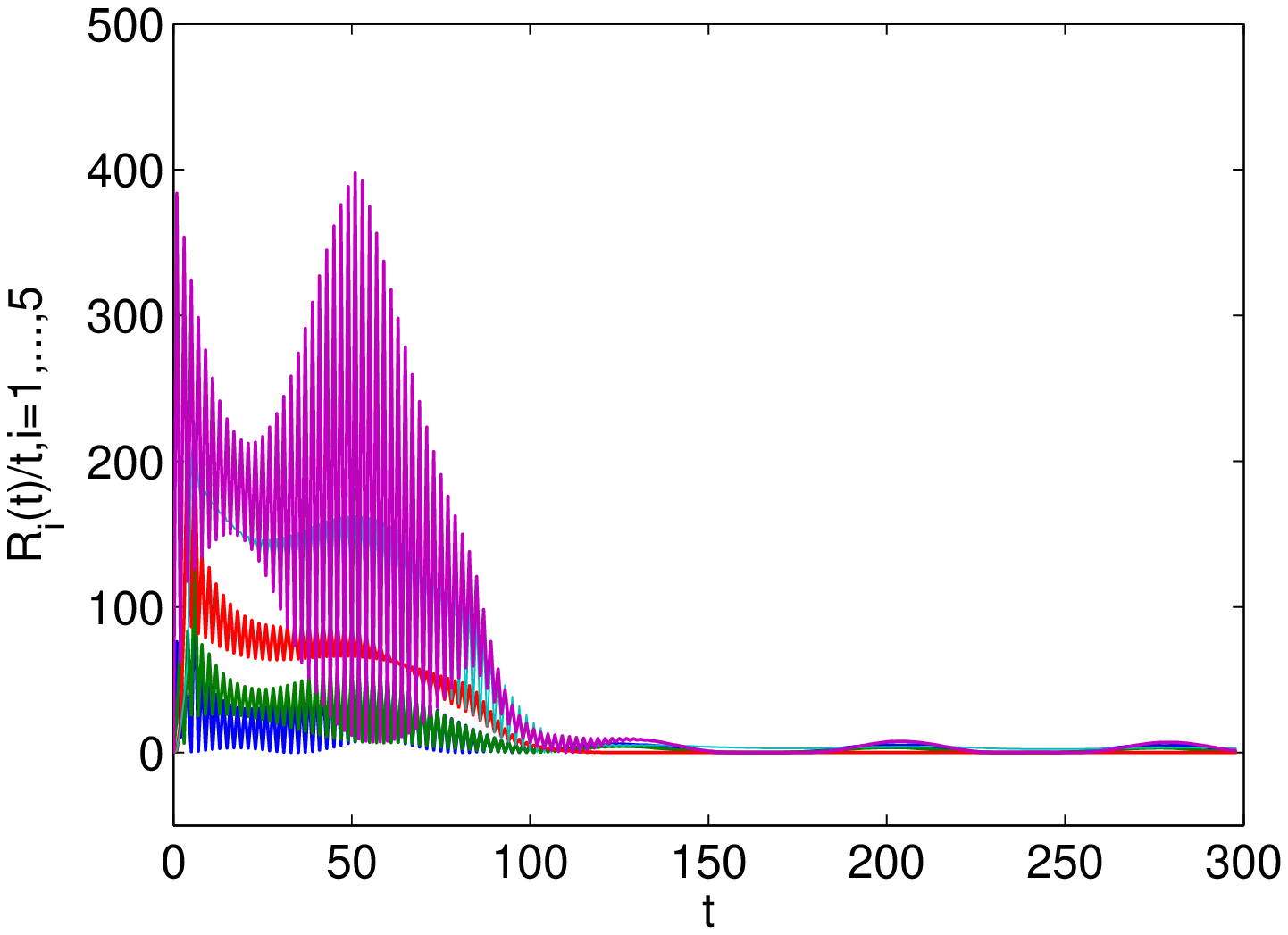}
\caption{The trajectory of ${\mathcal{R}_i}(t)/t$, $i=1,\cdots,5$.}\label{fig3}
\end{minipage}
\begin{minipage}[t]{0.5\linewidth}
\centering
\includegraphics[width=0.9\textwidth]{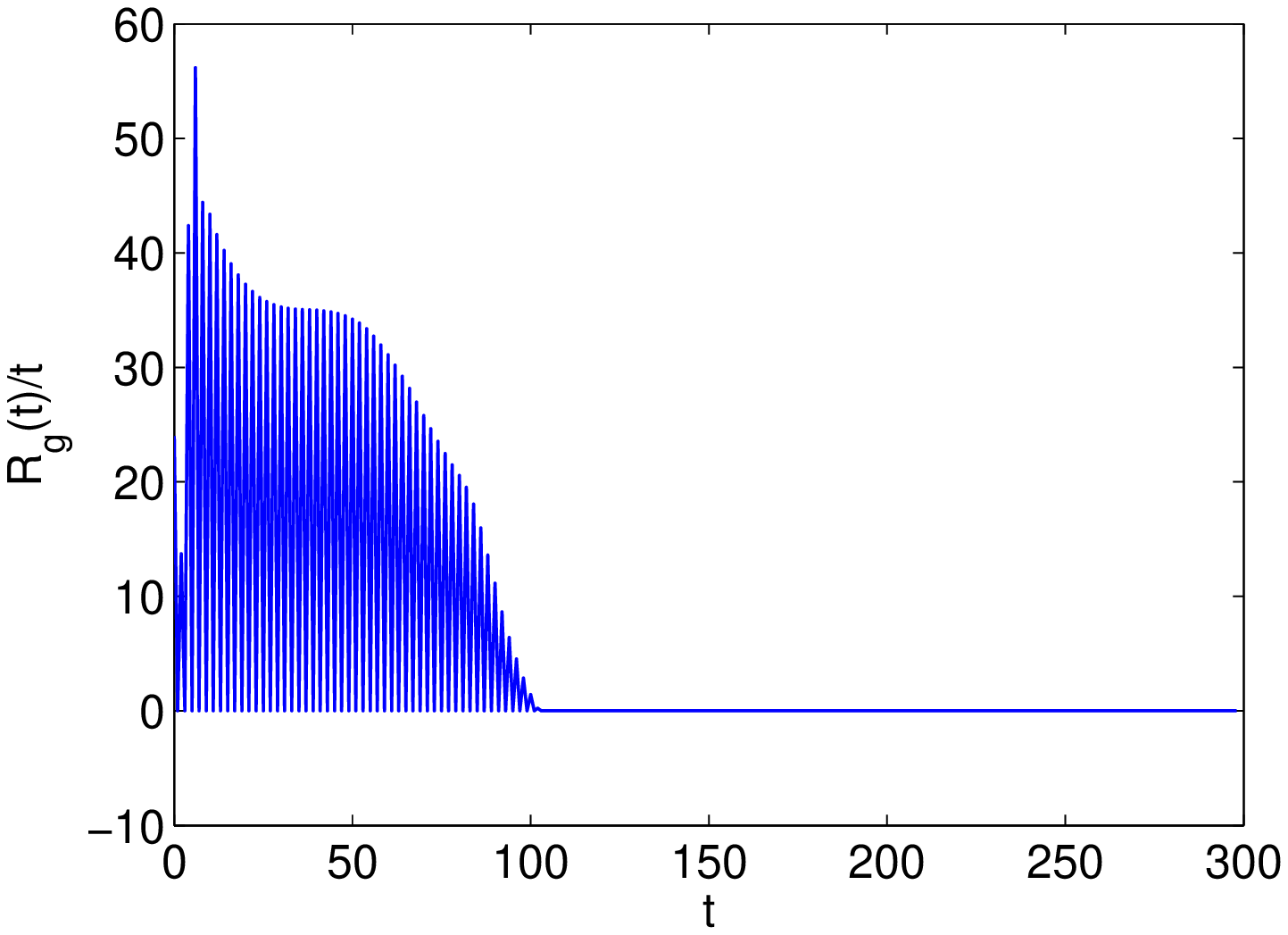}
\caption{The trajectory of ${\mathcal{R}_g}(t)/t$}\label{fig4}
\end{minipage}
\end{figure}

\section{Conclusions}\label{se4}

In this paper, an online distributed primal-dual algorithm has been presented for players to seek GNEs of non-cooperative games with dynamic cost functions, where all action sets are coupled through shared convex inequality constraints. By implementing the algorithm, each player makes decisions only using its own cost function, local set constraint, a local
block of the convex inequality constraint, and actions and estimates received
from its neighbors. The result shows that if the graph is connected, then the regrets, as well as the violation of inequality constraint, is bounded by product of a term depending on the deviation of variational GNE sequence and a sublinear function of learning time. A simulation example has been presented to demonstrate the effectiveness
of our theoretical results.
How to achieve lower bounds of the regrets and the violation is a difficult
problem in online non-cooperative games with shared inequality constraints. This topic will be considered in our future work.
 Our future work will also focus on some
other interesting topics, such as the case with time delays
and communication bandwidth constraints, which will bring new challenges in online distributed GNE seeking.

\section{appendix}
\subsection{Proof of Lemma \ref{le34}}\label{subse1}
Considering the norm of the third equation in (\ref{eq4}) and using the fact that $\sum_{j=1}^na_{ij}=\sum_{i=1}^na_{ij}=1$, we have

\[\begin{split}
&\sum_{i=1}^n\|y_i(t+1)\|^2\\
&\leq \sum_{i=1}^n\left\|\sum_{j=1}^na_{ij} ((1-\gamma^2(t)){y}_j(t)+\gamma(t)g_i(x_i(t)))\right\|^2\\
&\leq \sum_{i=1}^n\left(\sum_{j=1}^na_{ij}\left\| (1-\gamma^2(t)){y}_j(t)+\frac{\gamma^2(t)g_i(x_i(t))}{\gamma(t)}\right\|\right)^2\\
&\leq\sum_{i=1}^n\sum_{j=1}^na_{ij}\left((1-\gamma^2(t))\|{y}_j(t)\|^2+{\|g_i(x_i(t))\|^2}\right)\\
&\leq(1-\gamma^2(t))\sum_{i=1}^n \|{y}_i(t)\|^2+{n\kappa_2^2}
\end{split}\]
where the first inequality holds due to the non-expansive
property of projection, the second inequality results from triangle inequality, and the third inequality holds by using Jensen's inequality.
Because $ y_i(1)=0$, $\gamma^2(t)\leq1$ and $\gamma(t)\leq\gamma(t-1)$, it yields that
$\sum_{i=1}^n\|y_i(2)\|^2\leq\frac{n\kappa_2^2}{\gamma^2(1)}$; $\sum_{i=1}^n\|y_i(3)\|^2\leq (1-\gamma^2(2))\sum_{i=1}^n \|{y}_i(2)\|^2+{n\kappa_2^2}\leq\frac{n\kappa_2^2}{\gamma^2(2)}.$
In a similar fashion, we have $\sum_{i=1}^n\|y_i(t)\|^2\leq\frac{n\kappa_2^2}{\gamma^2(t-1)}$ for any $t\in\lfloor T\rfloor$. By the fact that $\gamma(t)$ is non-increasing, it leads to the validity of the result.\QEDA

\subsection{Proof of Lemma \ref{le5}}\label{subse9}
(i) By the first equation in (\ref{eq4}), we have
\begin{equation}\label{eq18}\begin{split}
e_{ih}(t+1)=&\sum_{k \in\mathcal{N}_i/\{h\}} a_{ik}{e}_{kh}(t)-(x_h(t+1)-x_h(t))
\end{split}\end{equation}
for any $i, h\in\mathcal{V}$. Let $\omega_h(t)=x_h(t)-x_h(t+1)$, by (\ref{eq18}), we have
\begin{equation}\label{eq20}\begin{split}
e_{h}(t+1)=(A_h^-\otimes I_m){e}_{h}(t)+\nu_h(t)~~~ h\in\mathcal{V}
\end{split}\end{equation}
where $\nu_h(t)=(\textbf{1}_{n-1}\otimes I_m)\omega_h(t)$.
By the second equation in (\ref{eq4}) and using the compactness of $\Omega_h$ in Assumption 2, we have
$\|\omega_h(t)\|\leq2\gamma(t)\kappa_1~~\emph{\emph{for any}}~h\in\mathcal{V}.$ By (\ref{eq20}), taking the norm of $e_{h}$ and replacing label $h$ with label $i$, we have
 \[\begin{split}
\|e_{i}(t+1)\|\leq \lambda\|{e}_{i}(t)\|+2\sqrt{n-1}\kappa_1\gamma(t)~~~ i\in\mathcal{V}
\end{split}\]
which immediately leads to the validity of (i).

(ii) By the facts that $x_{h}(1)\in\Omega_h$ and $x_{ih}(1)=0$ for any $i\neq h$, we know that $\|{e}_{i}(1)\|\leq \sqrt{n-1}\kappa_1$ for any $i\in\mathcal{V}$. Hence,
\[\begin{split}
\|e_{i}(t)\|^2&\leq \lambda^{t-1}\|{e}_{i}(1)\|^2+4(n-1)\kappa_1^2\sum_{k=0}^{t-2}\lambda^{k}\gamma(t-k-1)+4({n-1})\kappa_1^2\left(\sum_{k=0}^{t-2}\lambda^{k}\gamma(t-k-1)\right)^2\\
\end{split}\]
Using Cauchy-Schwarz inequality yields
\[\begin{split}
\left(\sum_{k=0}^{t-2}\lambda^{k}\gamma(t-k-1)\right)^2&\leq\left(\sum_{k=0}^{t-2}\lambda^{k}\right)\left(\sum_{k=0}^{t-2}\lambda^{k}(\gamma(t-k-1))^2\right)\\
&\leq\frac{\sum_{k=0}^{t-2}\lambda^{k}\gamma(t-k-1)}{1-\lambda}
\end{split}\]
where the second inequality results from the fact $\gamma(t)\leq1$. Above two inequalities lead to the validity of (ii).

(iii) Let $\theta_i(t)=\Big[(1-\gamma^2(t))\sum_{j\in {\mathcal{N}_i}}a_{ij}{y}_j(t)+\gamma(t)g_i(x_i(t))\Big]_+-\sum_{j\in {\mathcal{N}_i}}a_{ij}{y}_j(t)$, note that $\sum_{j\in {\mathcal{N}_i}}a_{ij}{y}_j(t)$ $\in\mathbb{R}^r_+$, thus,
\begin{equation}\label{t-31}\begin{split}
\|\theta_{i}(t)\|&\leq\left\|-\gamma(t)\sum_{j\in {\mathcal{N}_i}}a_{ij}\gamma(t){y}_j(t)+\gamma(t)g_i(x_i(t))\right\|\\
&\leq\gamma(t)\sum_{j\in {\mathcal{N}_i}}a_{ij}\left\|\gamma(t){y}_j(t)\right\|+\gamma(t)\left\|g_i(x_i(t))\right\|\\
&\leq({\sqrt{n}+1})\kappa_2\gamma(t)
\end{split}\end{equation}
where the first inequality results from the non-expansive property of projection, and the third inequality holds by using Lemma \ref{le34}. By the third equation in (\ref{eq4}), we have
\[\begin{split}
{y_i(t+1)} =\sum_{j\in {\mathcal{N}_i}}a_{ij}{y}_j(t)+\theta_i(t)~~~ i\in\mathcal{V}
\end{split}\]
Define $y=[y_1^T, \cdots, y_n^T]^T$ and $Y=y-(\textbf{1}_n\otimes I_r)\bar{y}$, we have
\begin{equation}\label{eq-31}
{Y}(t+1)=(A\otimes I_r){{Y}}(t)+\theta(t)
\end{equation}
where $\theta(t)=\left(\left(I-\frac{\textbf{1}\textbf{1}^T}{n}\right)\otimes I_r\right)[\theta_1(t), \cdots, \theta_n(t)]^T$.  Since $A$ is symmetric for $\mathcal{G}(A)$ being undirected, there must exist an orthogonal matrix $Q=\left[\frac{\textbf{1}_n}{\sqrt{n}}, Q_0\right]$ such that $Q^TLQ=\diag(1, \lambda_2(A), \cdots, \lambda_n(A))$. Define $\tilde{Y}(t)= (Q^T\otimes I_r)Y(t)$, it is obvious that $\tilde{Y}(0)=0$. By (\ref{eq-31}), it follows
\begin{equation*}\label{eq-31a}
 \tilde{Y}(t+1)=\diag(1, \lambda_2(A)I_r, \cdots, \lambda_n(A) I_r)\tilde{Y}(t)+(Q^T\otimes I_r)\theta(t).
\end{equation*}
Then, together with the facts that $ \Vert Q \Vert=1$ and $\Vert I-\textbf{1}\textbf{1}^T/n\Vert\leq2$, we have
\begin{equation}\label{T-33}\begin{split}
\|y-(\textbf{1}_n\otimes I_r)\bar{y}\|&= \Vert Q\tilde{Y}(t) \Vert\\
&\leq \Vert \tilde{Y}(t) \Vert\\
&\leq \sum_{k=0}^{t-2}(\sigma_2(A))^k\|\theta(t-1-k)\|
         \end{split}\end{equation}
By (\ref{t-31}), we have $\|\theta(t-1-k)\|\leq2(n+{\sqrt{n}})\kappa_2\gamma(t-1-k)$. Submitting it into (\ref{T-33}) yields inequality in (iii).

(iv) Similar to the proof of (ii), using Cauchy-Schwarz inequality, the validity of (iv) can be verified.\QEDA

\subsection{Proof of Lemma \ref{le2}}\label{subse2}
Note that by algorithm (\ref{eq4}), ${y}_i(t)$ and ${x}_i(t)$ will stay in $\mathbb{R}^r_+$ and $\in{\Omega_i}$ all the time, respectively. Note that
\begin{equation}\label{eq5}\begin{split}
&\sum_{i=1}^n\|x_i(t+1)-x_i^*(t+1)\|^2\\
&=\sum_{i=1}^n\|x_i(t+1)-x_i^*(t)\|^2+\sum_{i=1}^n\langle x_i^*(t+1)+x_i^*(t)\\
&~~~-2x_i(t+1),  x_i^*(t+1)-x_i^*(t)   \rangle\\
&\leq \sum_{i=1}^n\|x_i(t+1)-x_i^*(t)\|^2+4\kappa_1 \sum_{i=1}^n \|x_i^*(t+1)-x_i^*(t) \|.  \\
\end{split}\end{equation}
Moreover, by the second equation in algorithm (\ref{eq4}), we have
\begin{equation}\label{buc0}\begin{split}
\sum_{i=1}^n\|x_i(t+1)-x_i^*(t)\|^2&= \sum_{i=1}^n\Big\|(1-\gamma(t))(x_i(t)-x_i^*(t))+\gamma(t)\Big(P_{\Omega_i}[ {x}_i(t)\\
&~~~-\gamma(t)(\nabla_{x_i}J_i^t(\textbf{x}_i(t))+\gamma(t)\nabla_{x_i}g_i(x_i(t))y_i(t))]-x_i^*(t)\Big)\Big\|^2\\
\end{split}\end{equation}
where $\nabla_{x_i}g_i(x_i(t))=[\nabla_{x_i}g_{i1}(x_i(t)), \cdots, $ $\nabla_{x_i}g_{ir}(x_i(t))]$.
By (\ref{eq4jjk}), we have $x_i^*(t)={P_{{\Omega_i}}}[{x_i^*(t)}- \gamma(t)(\nabla_{x_i}J_i^t(x^*(t))+ \gamma(t)\nabla_{x_i}g_i(x_i^*(t))y^*(t))]$. Then, from (\ref{buc0}), it implies that
\begin{equation}\label{buc1}\begin{split}
\sum_{i=1}^n\|x_i(t+1)-x_i^*(t)\|^2&= \sum_{i=1}^n\Big\|(1-\gamma(t))(x_i(t)-x_i^*(t))+\gamma(t)\Big(P_{\Omega_i}[ {x}_i(t)\\
&~~~-\gamma(t)(\nabla_{x_i}J_i^t(\textbf{x}_i(t))+\gamma(t)\nabla_{x_i}g_i(x_i(t))y_i(t))]-{P_{{\Omega_i}}}[{x_i^*(t)} \\
  &~~~- \gamma(t)(\nabla_{x_i}J_i^t(x^*(t))+ \gamma(t)\nabla_{x_i}g_i(x_i^*(t))y^*(t))]\Big)\Big\|^2.\\
\end{split}\end{equation}
By Jensen's inequality, we know that $\|\gamma u+(1-\gamma) v\|^2\leq \gamma\|x\|^2+(1-\gamma) \|y\|^2$ for any $0\leq\gamma\leq1$ and any vectors $u, v\in \mathbb{R}^m$. From (\ref{buc1}), it follows that
\begin{equation}\label{buc2}\begin{split}
\sum_{i=1}^n\|x_i(t+1)-x_i^*(t)\|^2&\leq (1-\gamma(t))\sum_{i=1}^n\Big\|(x_i(t)-x_i^*(t))\Big\|^2+\gamma(t)\sum_{i=1}^n\Big\|P_{\Omega_i}[ {x}_i(t)\\
&~~~-\gamma(t)(\nabla_{x_i}J_i^t(\textbf{x}_i(t))+\gamma(t)\nabla_{x_i}g_i(x_i(t))y_i(t))]-{P_{{\Omega_i}}}[{x_i^*(t)} \\
  &~~~- \gamma(t)(\nabla_{x_i}J_i^t(x^*(t))+ \gamma(t)\nabla_{x_i}g_i(x_i^*(t))y^*(t))]\Big\|^2.\\
\end{split}\end{equation}
Using the non-expansive property of projection, one knows that $\|{P_{{\Omega_i}}}[u]-{P_{{\Omega_i}}}[v]\|\leq\|u-v\|$. Hence,
\begin{equation}\label{buc3}\begin{split}
&\Big\|P_{\Omega_i}[ {x}_i(t)-\gamma(t)(\nabla_{x_i}J_i^t(\textbf{x}_i(t))+\gamma(t)\nabla_{x_i}g_i(x_i(t))y_i(t))]\\
&-{P_{{\Omega_i}}}[{x_i^*(t)}- \gamma(t)(\nabla_{x_i}J_i^t(x^*(t))+ \gamma(t)\nabla_{x_i}g_i(x_i^*(t))y^*(t))]\Big\|^2\\
&\leq\Big\|({x}_i(t)-x_i^*(t))-\gamma(t)\left(\nabla_{x_i}J_i^t(\textbf{x}_i(t))-\nabla_{x_i}J_i^t(x^*(t))\right)
-\gamma^2(t)\left(\nabla_{x_i}g_i(x_i(t))y_i(t)-\nabla_{x_i}g_i(x_i^*(t))y^*(t)\right)\Big\|^2\\
&=\|{x}_i(t)-x_i^*(t)\|^2+\gamma^2(t)\|\nabla_{x_i}J_i^t(\textbf{x}_i(t))-\nabla_{x_i}J_i^t(x^*(t))\|^2+
\gamma^4(t)\|\nabla_{x_i}g_i(x_i(t))y_i(t)-\nabla_{x_i}g_i(x_i^*(t))y^*(t)\|^2\\
&~~~-2\gamma(t)\langle {x}_i(t)-x_i^*(t), \nabla_{x_i}J_i^t(\textbf{x}_i(t))-\nabla_{x_i}J_i^t(x^*(t))\rangle\\
&~~~-2\gamma^2(t)\langle {x}_i(t)-x_i^*(t), \nabla_{x_i}g_i(x_i(t))y_i(t)-\nabla_{x_i}g_i(x_i^*(t))y^*(t)\rangle\\
&~~~+2\gamma^3(t)\langle \nabla_{x_i}J_i^t(\textbf{x}_i(t))-\nabla_{x_i}J_i^t(x^*(t)), \nabla_{x_i}g_i(x_i(t))y_i(t)-\nabla_{x_i}g_i(x_i^*(t))y^*(t)\rangle.\\
\end{split}\end{equation}
By the fact that $\|u+v\|^2\leq 2(\|u\|^2+\|v\|^2)$, we have
\[\label{eqlulu5}\begin{split}
&\gamma^4(t)\|\nabla_{x_i}g_i(x_i(t))y_i(t)-\nabla_{x_i}g_i(x_i^*(t))y^*(t)\|^2\\
&\leq 2\gamma^4(t)\|\nabla_{x_i}g_i(x_i(t))y_i(t)\|^2+ 2\gamma^4(t)\|\nabla_{x_i}g_i(x_i^*(t))y^*(t)\|^2\\
&\leq 2\gamma^4(t)\|\nabla_{x_i}g_i(x_i(t))y_i(t)\|^2+ 2\gamma^2(t)\|\nabla_{x_i}g_i(x_i^*(t))y^*(t)\|^2\\
&\leq 2\gamma^2\|\nabla_{x_i}g_i(x_i(t))\|^2\|\gamma(t)y_i(t)\|^2+ 2\gamma^2(t)\|\nabla_{x_i}g_i(x_i^*(t))\|\|y^*(t)\|^2\\
&\leq 2\gamma^2(t)\kappa_0^2\|\gamma(t)y_i(t)\|^2+ 2\theta^2\kappa_0^2\gamma^2(t)\\
\end{split}\]
where the second inequality results from the fact that $0\leq\gamma(t)\leq1$. Note that
\[\label{eqlulu5}\begin{split}
&2\gamma^3(t)\langle \nabla_{x_i}J_i^t(\textbf{x}_i(t))-\nabla_{x_i}J_i^t(x^*(t)), \nabla_{x_i}g_i(x_i(t))y_i(t)-\nabla_{x_i}g_i(x_i^*(t))y^*(t)\rangle\\
&\leq2\gamma^2(t)\|\nabla_{x_i}J_i^t(\textbf{x}_i(t))-\nabla_{x_i}J_i^t(x^*(t))\|\Big(\|\nabla_{x_i}g_i(x_i(t))\|\|\gamma(t)y_i(t)\|
+\gamma(t)\|\nabla_{x_i}g_i(x_i^*(t))\|y^*(t)\|\Big)\\
&\leq2\gamma^2(t)\|\nabla_{x_i}J_i^t(\textbf{x}_i(t))-\nabla_{x_i}J_i^t(x^*(t))\|\Big(\kappa_0\theta
+\kappa_0\|y^*(t)\|\Big)\\
&=2\kappa_0\gamma^2(t)\|\left(\nabla_{x_i}J_i^t(\textbf{x}_i(t))-\nabla_{x_i}J_i^t({x}(t))\right)
+\left(\nabla_{x_i}J_i^t({x}(t))-\nabla_{x_i}J_i^t(x^*(t))\right)\|\Big(\|\gamma(t)y_i(t)\|
+\theta\Big)\\
&\leq2\kappa_0\gamma^2(t)(\|\gamma(t)y_i(t)\|
+\theta)\|\nabla_{x_i}J_i^t(\textbf{x}_i(t))-\nabla_{x_i}J_i^t({x}(t))\|+\\
&~~~2\kappa_0\gamma^2(t)\Big(\|\nabla_{x_i}J_i^t({x}(t))\|+\|\nabla_{x_i}J_i^t(x^*(t))\|\Big)\Big(\|\gamma(t)y_i(t)\|
+\theta\Big)\\
&\leq2\kappa_0\gamma^2(t)\Big(\|\gamma(t)y_i(t)\|
+\theta\Big)\|\nabla_{x_i}J_i^t(\textbf{x}_i(t))-\nabla_{x_i}J_i^t({x}(t))\|+4\kappa_0\kappa_3\gamma^2(t)\Big(\|\gamma(t)y_i(t)\|
+\theta\Big)\\
&\leq\kappa_0\gamma^2(t)(\|\gamma(t)y_i(t)\|
+\theta)^2+\kappa_0\gamma^2(t)\|\nabla_{x_i}J_i^t(\textbf{x}_i(t))-\nabla_{x_i}J_i^t({x}(t))\|^2+4\kappa_0\kappa_3\gamma^2(t)\Big(\|\gamma(t)y_i(t)\|
+\theta\Big)\\
\end{split}\]
where the last inequality holds by using the fact that $2ab\leq a^2+b^2$ for any $a, b\in\mathbb{R}$. Due to the facts that ${x}(t)=(x_i(t), x_{-i}(t))$ and ${x}(t), x^*(t)\in\Omega$, there holds $\|\nabla_{x_i}J_i^t({x}(t))\|\leq\kappa_3$ and $\|\nabla_{x_i}J_i^t(x^*(t))\|\leq\kappa_3$. Then, we have
\[\label{eqlulu5}\begin{split}
&\|\nabla_{x_i}J_i^t(\textbf{x}_i(t))-\nabla_{x_i}J_i^t(x^*(t))\|^2\\
&=\|\left(\nabla_{x_i}J_i^t(\textbf{x}_i(t))-\nabla_{x_i}J_i^t({x}(t))\right)
+\left(\nabla_{x_i}J_i^t({x}(t))-\nabla_{x_i}J_i^t(x^*(t))\right)\|^2\\
&\leq\|\nabla_{x_i}J_i^t(\textbf{x}_i(t))-\nabla_{x_i}J_i^t({x}(t))\|^2+\|\nabla_{x_i}J_i^t({x}(t))-\nabla_{x_i}J_i^t(x^*(t))\|^2\\
&~~~+2\|\nabla_{x_i}J_i^t(\textbf{x}_i(t))-\nabla_{x_i}J_i^t({x}(t))\|\|\nabla_{x_i}J_i^t({x}(t))-\nabla_{x_i}J_i^t(x^*(t))\|\\
&\leq\|\nabla_{x_i}J_i^t(\textbf{x}_i(t))-\nabla_{x_i}J_i^t({x}(t))\|^2+2\|\nabla_{x_i}J_i^t({x}(t))\|^2+2\|\nabla_{x_i}J_i^t(x^*(t))\|^2\\
&~~~+2\|\nabla_{x_i}J_i^t(\textbf{x}_i(t))-\nabla_{x_i}J_i^t({x}(t))\|(\|\nabla_{x_i}J_i^t({x}(t))\|+\|\nabla_{x_i}J_i^t(x^*(t))\|)\\
&\leq\|\nabla_{x_i}J_i^t(\textbf{x}_i(t))-\nabla_{x_i}J_i^t({x}(t))\|^2+4\kappa_3\|\nabla_{x_i}J_i^t(\textbf{x}_i(t))
-\nabla_{x_i}J_i^t({x}(t))\|+4\kappa_3^2\\
\end{split}\]
and
\[\label{eqlulu5}\begin{split}
&-\langle {x}_i(t)-x_i^*(t), \nabla_{x_i}J_i^t(\textbf{x}_i(t))-\nabla_{x_i}J_i^t(x^*(t))\rangle\\
&=-\langle {x}_i(t)-x_i^*(t), \nabla_{x_i}J_i^t({x}(t))-\nabla_{x_i}J_i^t(x^*(t))\rangle-\langle {x}_i(t)-x_i^*(t), \nabla_{x_i}J_i^t(\textbf{x}_i(t))-\nabla_{x_i}J_i^t(x(t))\rangle\\
&\leq-\langle {x}_i(t)-x_i^*(t), \nabla_{x_i}J_i^t({x}(t))-\nabla_{x_i}J_i^t(x^*(t))\rangle+\|{x}_i(t)-x_i^*(t)\| \|\nabla_{x_i}J_i^t(\textbf{x}_i(t))-\nabla_{x_i}J_i^t(x(t))\|\\
&\leq-\langle {x}_i(t)-x_i^*(t), \nabla_{x_i}J_i^t({x}(t))-\nabla_{x_i}J_i^t(x^*(t))\rangle+(\|{x}_i(t)\|+\|x_i^*(t)\|)\|\nabla_{x_i}J_i^t(\textbf{x}_i(t))-\nabla_{x_i}J_i^t(x(t))\|\\
&\leq-\langle {x}_i(t)-x_i^*(t), \nabla_{x_i}J_i^t({x}(t))-\nabla_{x_i}J_i^t(x^*(t))\rangle+
2\kappa_1\|\nabla_{x_i}J_i^t(\textbf{x}_i(t))-\nabla_{x_i}J_i^t(x(t))\|.\\
\end{split}\]
Then, by inequality (\ref{buc3}), it implies that
\begin{equation}\label{buc4}\begin{split}
&\Big\|P_{\Omega_i}[ {x}_i(t)-\gamma(t)(\nabla_{x_i}J_i^t(\textbf{x}_i(t))+\gamma(t)\nabla_{x_i}g_i(x_i(t))y_i(t))]-\\
&{P_{{\Omega_i}}}[{x_i^*(t)}- \gamma(t)(\nabla_{x_i}J_i^t(x^*(t))+ \gamma(t)\nabla_{x_i}g_i(x_i^*(t))y^*(t))]\Big\|^2\\
&\leq\|{x}_i(t)-x_i^*(t)\|^2+\gamma^2(t)\Big(\|\nabla_{x_i}J_i^t(\textbf{x}_i(t))-\nabla_{x_i}J_i^t({x}(t))\|^2\\
&~~~+4\kappa_3\|\nabla_{x_i}J_i^t(\textbf{x}_i(t))
-\nabla_{x_i}J_i^t({x}(t))\|+4\kappa_3^2\Big)\\
&~~~+2\gamma^2(t)\kappa_0^2\|\gamma(t)y_i(t)\|^2+ 2\theta^2\kappa_0^2\gamma^2(t)\\
&~~~-2\gamma(t)\Big(\langle {x}_i(t)-x_i^*(t), \nabla_{x_i}J_i^t({x}(t))-\nabla_{x_i}J_i^t(x^*(t))\rangle+
2\kappa_1\|\nabla_{x_i}J_i^t(\textbf{x}_i(t))-\nabla_{x_i}J_i^t(x(t))\|\Big)\\
&~~~+\kappa_0\gamma^2(t)(\|\gamma(t)y_i(t)\|
+\theta)^2+\kappa_0\gamma^2(t)\|\nabla_{x_i}J_i^t(\textbf{x}_i(t))-\nabla_{x_i}J_i^t({x}(t))\|^2\\
&~~~+4\kappa_0\kappa_3\gamma^2(t)\Big(\|\gamma(t)y_i(t)\|
+\theta\Big)-2\gamma^2(t)\langle {x}_i(t)-x_i^*(t), \nabla_{x_i}g_i(x_i(t))y_i(t)-\nabla_{x_i}g_i(x_i^*(t))y^*(t)\rangle\\
&=\|{x}_i(t)-x_i^*(t)\|^2+\gamma^2(t)\Big((1+\kappa_0)\|\nabla_{x_i}J_i^t(\textbf{x}_i(t))-\nabla_{x_i}J_i^t({x}(t))\|^2\\
&~~~+4(\kappa_1+\kappa_3)\|\nabla_{x_i}J_i^t(\textbf{x}_i(t))
-\nabla_{x_i}J_i^t({x}(t))\|\Big)\\
&~~~+2\gamma^2(t)\kappa_0^2\|\gamma(t)y_i(t)\|^2+ \Big(2\theta^2\kappa_0^2+4\kappa_3^2+4\kappa_0\kappa_3\theta\Big)\gamma^2(t)\\
&~~~-2\gamma(t)\Big\langle {x}_i(t)-x_i^*(t), \nabla_{x_i}J_i^t({x}(t))-\nabla_{x_i}J_i^t(x^*(t))\Big\rangle\\
&~~~+\kappa_0\gamma^2(t)(\|\gamma(t)y_i(t)\|
+\theta)^2\\
&~~~+4\kappa_0\kappa_3\gamma^2(t)\|\gamma(t)y_i(t)\|-2\gamma^2(t)\langle {x}_i(t)-x_i^*(t), \nabla_{x_i}g_i(x_i(t))y_i(t)-\nabla_{x_i}g_i(x_i^*(t))y^*(t)\rangle.\\
\end{split}\end{equation}
It is noticed that
\begin{equation}\label{buc5}\begin{split}
& -\langle {x}_i(t)-x_i^*(t), \nabla_{x_i}g_i(x_i(t))y_i(t)-\nabla_{x_i}g_i(x_i^*(t))y^*(t)\rangle\\
&=-\langle {x}_i(t)-x_i^*(t), \nabla_{x_i}g_i(x_i(t))y_i(t)\rangle+\langle {x}_i(t)-x_i^*(t), \nabla_{x_i}g_i(x_i^*(t))y^*(t)\rangle\\
&\leq-\langle {x}_i(t)-x_i^*(t), \nabla_{x_i}g_i(x_i(t))y_i(t)\rangle+\|{x}_i(t)-x_i^*(t)\| \|\nabla_{x_i}g_i(x_i^*(t))\|y^*(t)\|\\
&\leq-\langle {x}_i(t)-x_i^*(t), \nabla_{x_i}g_i(x_i(t))y_i(t)\rangle+2\theta\kappa_1\kappa_3.
\end{split}\end{equation}
Based on (\ref{buc4}) and (\ref{buc5}), it follows that
\begin{equation}\label{buc6}\begin{split}
&\Big\|P_{\Omega_i}[ {x}_i(t)-\gamma(t)(\nabla_{x_i}J_i^t(\textbf{x}_i(t))+\gamma(t)\nabla_{x_i}g_i(x_i(t))y_i(t))]\\
&-{P_{{\Omega_i}}}[{x_i^*(t)}- \gamma(t)(\nabla_{x_i}J_i^t(x^*(t))+ \gamma(t)\nabla_{x_i}g_i(x_i^*(t))y^*(t))]\Big\|^2\\
&\leq\|{x}_i(t)-x_i^*(t)\|^2+\gamma^2(t)\Big((1+\kappa_0)\|\nabla_{x_i}J_i^t(\textbf{x}_i(t))-\nabla_{x_i}J_i^t({x}(t))\|^2\\
&~~~+4(\kappa_1+\kappa_3)\|\nabla_{x_i}J_i^t(\textbf{x}_i(t))
-\nabla_{x_i}J_i^t({x}(t))\|\Big)\\
&~~~+2\gamma^2(t)\kappa_0^2\|\gamma(t)y_i(t)\|^2+ \Big(2\theta^2\kappa_0^2+4\kappa_3^2+4\kappa_0\kappa_3\theta+4\theta\kappa_1\kappa_3\Big)\gamma^2(t)\\
&~~~-2\gamma(t)\Big\langle {x}_i(t)-x_i^*(t), \nabla_{x_i}J_i^t({x}(t))-\nabla_{x_i}J_i^t(x^*(t))\Big\rangle\\
&~~~+\kappa_0\gamma^2(t)(\|\gamma(t)y_i(t)\|
+\theta)^2+4\kappa_0\kappa_3\gamma^2(t)\|\gamma(t)y_i(t)\|\\
&~~~-2\gamma^2(t)\Big\langle {x}_i(t)-x_i^*(t), \nabla_{x_i}g_i(x_i(t))y_i(t)\Big\rangle.\\
\end{split}\end{equation}
It is not difficult to verify that $\tau=4n\kappa_0\theta(\theta\kappa_0+\kappa_1+\kappa_3)+4n\kappa_3^2\geq2\theta^2\kappa_0^2+4\kappa_3^2+4\kappa_0\kappa_3\theta+4\theta\kappa_1\kappa_3$,
submitting (\ref{buc6}) into (\ref{buc2}) and using the fact that $\sum_{i=1}^{n}\Big\langle {x}_i(t)-x_i^*(t), \nabla_{x_i}J_i^t({x}(t))-\nabla_{x_i}J_i^t(x^*(t))\Big\rangle=\Big\langle {x}(t)-x^*(t), F^t({x}(t))-F^t(x^*(t))\Big\rangle$, we have
\begin{equation}\label{buc7}\begin{split}
\sum_{i=1}^n\|x_i(t+1)-x_i^*(t)\|^2&\leq\sum_{i=1}^n\|x_i(t)-x_i^*(t)\|^2+2\gamma^3(t)\kappa_0^2\sum_{i=1}^n\|\gamma(t)y_i(t)\|^2\\
 &~~~+\kappa_0\gamma^3(t)\sum_{i=1}^n(\theta+\|\gamma(t)y_i(t)\|)^2+\tau\gamma^3(t)\\
  &~~~-2\gamma^2(t)\langle x(t)-x^*(t), F^t(x(t))- F^t(x^*(t))\rangle\\
   &~~~+ (1+\kappa_0)\gamma^2(t)\sum_{i=1}^n\|\nabla_{x_i}J_i^t(\textbf{x}_i(t))-\nabla_{x_i}J_i^t({x}(t))\|^2\\
  &~~~+4(\kappa_1+\kappa_3)\gamma^2(t)\sum_{i=1}^n\|\nabla_{x_i}J_i^t(\textbf{x}_i(t))-\nabla_{x_i}J_i^t({x}(t))\|\\
&~~~-2\gamma^3(t)\sum_{i=1}^n\langle x_i(t)-x_i^*(t), \nabla_{x_i}g_i(x_i(t))y_i(t)\rangle+4\kappa_0\kappa_3\gamma^3(t)\sum_{i=1}^n\|\gamma(t)y_i(t)\|.
\end{split}\end{equation}
By condition (i) in Assumption 2, we have $$\langle x(t)-x^*(t), F^t(x(t))- F^t(x^*(t))\rangle\geq\mu\|x(t)-x^*(t)\|^2.$$  By condition (ii) in Assumption 2, it yields that
 \[\label{eq6}\begin{split}
\sum_{i=1}^n\|\nabla_{x_i}J_i^t(\textbf{x}_i(t))-\nabla_{x_i}J_i^t({x}(t))\|&\leq\ell\sum_{i=1}^n\|\textbf{x}_{-i}(t)-{x}_{-i}(t)\|\\
&=\ell\sqrt{\left(\sum_{i=1}^n\|\textbf{x}_{-i}(t)-{x}_{-i}(t)\|\right)^2}\\
&\leq\ell\sqrt{n\sum_{i=1}^n\|\textbf{x}_{-i}(t)-{x}_{-i}(t)\|^2}\\
\end{split}\]
\[\label{eq6}\begin{split}
&=\ell\sqrt{(n-1)\sum_{i=1}^n\|e_{i}(t)\|^2}\\
&\leq \ell\sqrt{n-1}\sum_{i=1}^n\|e_i(t)\|\\
\end{split}\]
and $$\sum_{i=1}^n\|\nabla_{x_i}J_i^t(\textbf{x}_i(t))-\nabla_{x_i}J_i^t({x}(t))\|^2\leq \ell^2\sum_{i=1}^n\|e_i(t)\|^2.$$
By Lemma 1, one has $\|\gamma(t)y_i(t)\|\leq{\sqrt{n}\kappa_2}.$ Then, from (\ref{buc7}), it follows that
\[\label{eqlulu5}\begin{split}
\sum_{i=1}^n\|x_i(t+1)-x_i^*(t)\|^2&\leq\sum_{i=1}^n\|x_i(t)-x_i^*(t)\|^2+(\tau+2\kappa_0^2{n}^2\kappa_2^2
 +4\kappa_0\kappa_3n\sqrt{n}\kappa_2+\kappa_0n(\theta+\sqrt{n}\kappa_2)^2)\gamma^3(t)\\
  &~~~-2\mu\gamma^2(t)\|x(t)-x^*(t)\|^2+ (1+\kappa_0)\ell^2\gamma^2(t)\sum_{i=1}^n\|e_i(t)\|^2\\
  &~~~+4\sqrt{n-1}(\kappa_1+\kappa_3)\ell\gamma^2(t)\sum_{i=1}^n\|e_i(t)\|-2\gamma^3(t)\sum_{i=1}^n\langle x_i(t)-x_i^*(t), \nabla_{x_i}g_i(x_i(t))y_i(t)\rangle.
\end{split}\]
By definitions of $\pi_i, i=1, 2$, we have
$$2\pi_1=\tau+2n^2(\kappa_0^2+\kappa_0)\kappa_2^2+4\kappa_0\kappa_3\kappa_2n\sqrt{n}+2n\kappa_0\theta^2
\geq\tau+2\kappa_0^2{n}^2\kappa_2^2
 +4\kappa_0\kappa_2\kappa_3n\sqrt{n}+\kappa_0n(\theta+\sqrt{n}\kappa_2)^2$$ and $2\pi_2=4\sqrt{n-1}\ell(\kappa_1+\kappa_3)$. Thus, 
 \begin{equation}\label{eq6}\begin{split}
& \sum_{i=1}^n\|x_i(t+1)-x_i^*(t)\|^2 \\
&\leq\sum_{i=1}^n\|x_i(t)-x_i^*(t)\|^2+2\pi_1\gamma^3(t)-2\gamma^2(t)\mu\|x(t)\\
 &~~~-x^*(t)\|^2+ (1+\kappa_0)\ell^2\gamma^2(t)\sum_{i=1}^n\|e_i(t)\|^2\\
&~~~+2\pi_2\gamma^2(t)\sum_{i=1}^n\|e_i(t)\|-2\gamma^3(t)\sum_{i=1}^n\langle x_i(t)\\
&~~~-x_i^*(t), \nabla_{x_i}g_i(x_i(t))y_i(t)\rangle.\\
\end{split}\end{equation}
 The convexity of $g_i(x_i(t))y_i(t)$ with respect to $x_i$ implies that
\begin{equation}\label{eq8}\begin{split}
&\sum_{i=1}^n\langle x_i(t)-x_i^*, \nabla_{x_i}g_i(x_i(t))y_i(t)\rangle\\
&\geq\sum_{i=1}^n (g_i(x_i(t)))^Ty_i(t)-\sum_{i=1}^n(g_i(x_i^*(t)))^T(y_i(t)\\
&~~~-\bar{y}(t))-\sum_{i=1}^n(g_i(x_i^*(t)))^T\bar{y}(t)\\
&\geq\sum_{i=1}^n (g_i(x_i(t)))^Ty_i(t)-\kappa_2\sum_{i=1}^n\|y_i(t)-\bar{y}(t)\|
\end{split}\end{equation}
where the second inequality holds due to the facts that $\sum_{i=1}^ng_i(x_i^*)\leq0$ and $\bar{y}(t)\geq0$.
Submitting (\ref{eq6}) and (\ref{eq8}) into (\ref{eq5}), and taking the summation with respect to $t = 1, \cdots, T$, we have
\[\begin{split}
&\sum_{t=1}^T\|x(t)-x^*(t)\|^2\\
&\leq\sum_{t=1}^T\frac{1}{2\mu\gamma^2(t)}\Big(\sum_{i=1}^n\|x_i(t)-x_i^*(t)\|^2-\sum_{i=1}^n\|x_i(t+1)\\
&~~~-x_i^*(t+1)\|^2\Big)+\frac{\pi_1}{\mu}\sum_{t=1}^T\gamma(t)+ \frac{(1+\kappa_0)\ell^2}{2\mu}\sum_{t=1}^T\sum_{i=1}^n\|e_i(t)\|^2\\
&~~~+\frac{\pi_2}{\mu}\sum_{t=1}^T\sum_{i=1}^n\|e_i(t)\|-\frac{1}{\mu}\sum_{t=1}^T\sum_{i=1}^n \gamma(t)(g_i(x_i(t)))^Ty_i(t)\\
&~~~+\frac{2\kappa_1}{\mu\gamma^2(T)} \sum_{t=1}^T\sum_{i=1}^n \|x_i^*(t+1)-x_i^*(t) \|+\frac{\kappa_2}{\mu}\sum_{i=1}^n\|y_i(t)-\bar{y}(t)\|.
\end{split}\]
Note that $\|x_i(t)-x_i^*(t)\|^2\leq4\kappa_1^2$ for any $i\in\mathcal{V}$, therefore,
\[\begin{split}
&\sum_{t=1}^T\frac{1}{2\gamma^2(t)}\Big(\sum_{i=1}^n\|x_i(t)-x_i^*(t)\|^2-\sum_{i=1}^n\|x_i(t+1)-x_i^*(t+1)\|^2\Big)\\
&=\frac{1}{2\gamma^2(1)}\sum_{i=1}^n\|x_i(1)-x_i^*(1)\|^2-\frac{1}{2\gamma^2(T)}\sum_{i=1}^n\|x_i(T+1)\\
&~~~-x_i^*(T+1)\|^2+\frac{1}{2}\sum_{t=2}^T\Big(\frac{1}{\gamma^2(t)}-\frac{1}{\gamma^2(t-1)}\Big)\sum_{i=1}^n\|x_i(t)-x_i^*(t)\|^2\\
&\leq\frac{2n\kappa_1^2}{\gamma^2(T)}.
\end{split}\]
 Using above two inequalities and the fact that $\sum_{t=1}^T\sum_{i=1}^n \|x_i^*(t+1)-x_i^*(t) \|\leq\sqrt{n}\Theta_T$ leads to the validity of inequality (\ref{eq9}).\QEDA

\subsection{Proof of Lemma \ref{le3}}\label{subse1}
By the third equation in (\ref{eq4}), the following recursion for Lagrange multipliers holds
 \begin{equation}\label{eq10}\begin{split}
&\sum_{i=1}^n\|y_i(t+1)-y\|^2\\
 &= \sum_{i=1}^n\left\|\left[(1-\gamma^2(t))\sum_{j\in {\mathcal{N}_i}}a_{ij}{y}_j(t)+\gamma(t)g_i(x_i(t))\right]_+-y\right\|^2\\
&\leq\sum_{i=1}^n\Big\|(1-\gamma^2(t))\sum_{j\in {\mathcal{N}_i}}a_{ij}({y}_j(t)-{y}_i(t))+({y}_i(t)-y)+\gamma(t)(g_i(x_i(t))-\gamma(t){y}_i(t))\Big\|^2\\
&\leq\sum_{i=1}^n\Big(\sum_{j\in {\mathcal{N}_i}}a_{ij}\|{y}_j(t)-{y}_i(t)\|^2+\|{y}_i(t)-y\|^2\\
&~~~+\gamma^2(t)\|g_i(x_i(t))-\gamma(t){y}_i(t)\|^2+2(1-\gamma^2(t))\sum_{j\in {\mathcal{N}_i}}a_{ij}\langle {y}_j(t)-{y}_i(t), {y}_i(t)\\
&~~~-y\rangle+2\gamma(t)\sum_{j\in {\mathcal{N}_i}}a_{ij}\|{y}_j(t)-{y}_i(t)\|\|g_i(x_i(t))-\gamma(t){y}_i(t)\|+2\gamma(t)\langle{y}_i(t)-y, g_i(x_i(t))-\gamma(t){y}_i(t)\rangle\Big)\\
\end{split}\end{equation}
where the first inequality results from the non-expansive
property of projection, and the second inequality holds by using Jensen's inequality and the fact that $\gamma(t)<1$. Moreover,
\begin{equation}\label{eq11}\begin{split}
&\sum_{i=1}^n\sum_{j=1}^na_{ij}\langle {y}_j(t)-{y}_i(t), {y}_i(t)-y\rangle\\
&=\sum_{i=1}^n\sum_{j=1}^na_{ij}\langle{y}_j(t)-{y}_i(t), {y}_i(t)-\bar{y}(t)\rangle\\
&\leq\frac{1}{2}\sum_{i=1}^n\sum_{j=1}^na_{ij}( \|y_j(t)-{y}_i(t)\|^2+\| {y}_i(t)-\bar{y}(t)\|^2)\\
\end{split}\end{equation}
where the equation holds due to the fact that $\sum_{i=1}^n\sum_{j=1}^n a_{ij}\langle {y}_j(t)$ $-{y}_i(t), z\rangle=0$ for any $z\in\mathbb{R}^r$, and the inequality holds by using Young's inequality. Using $(a-b)^2\leq2(a^2+b^2)$, we have
\begin{equation}\label{eq13}\begin{split}
&\sum_{i=1}^n\sum_{j=1}^na_{ij}( \|y_j(t)-{y}_i(t)\|^2)\leq4\sum_{i=1}^n(\| {y}_i(t)-\bar{y}(t)\|^2)\\
\end{split}\end{equation}
It is noticed that
\begin{equation}\label{eqh}\begin{split}
&\langle{y}_i(t)-y, g_i(x_i(t))-\gamma(t){y}_i(t)\rangle\leq \langle g_i(x_i(t)), y_i(t)-y\rangle+\frac{{\gamma(t)(\|y\|}^2-{\|y_i(t)\|}^2}{2}
\end{split}\end{equation}
Using Lemma \ref{le34}, we know that $\|g_i(x_i(t))-\gamma(t){y}_i(t)\|\leq \kappa_2+\sqrt{n}\kappa_2$ and $\|g_i(x_i(t))-\gamma(t){y}_i(t)\|^2\leq 2\kappa_2^2+2{n}\kappa_2^2$. Using inequalities (\ref{eq10})-(\ref{eqh}) and taking the summation with respect to $t = 1, \cdots, T$, we have
\begin{equation}\label{eqcC}\begin{split}
&\sum_{t=1}^T\sum_{i=1}^n\gamma(t)\left((g_i(x_i(t)))^T{y}_i(t)-(g_i(x_i(t)))^T{y}+\frac{{\gamma(t)\|y\|}^2}{2}\right)\\
&\geq\sum_{t=1}^T\frac{1}{2}\left(\sum_{i=1}^n\|y_i(t+1)-y\|^2-\sum_{i=1}^n\|{y}_i(t)-y\|^2\right)\\
&~~~-\sum_{t=1}^T\frac{9}{2}\sum_{i=1}^n\| {y}_i(t)-\bar{y}(t)\|^2- n\kappa_2^2(1+n)\sum_{t=1}^T\gamma^2(t)\\
&~~~-\kappa_2(1+\sqrt{n})\sum_{t=1}^T\gamma(t)\sum_{i=1}^n\sum_{j=1}^na_{ij}\|{y}_j(t)-{y}_i(t)\|.
\end{split}\end{equation}
Note that $y_i(1)=0$, then
\begin{equation}\label{eqd}\begin{split}
&\sum_{t=1}^T\left(\|y_i(t+1)-y\|^2-\|{y}_i(t)-y\|^2\right)\geq-{\|y\|^2}.
\end{split}\end{equation}
By (\ref{eqcC}) and (\ref{eqd}), together with the facts that $\sum_{i=1}^n\sum_{j=1}^na_{ij}\|{y}_j(t)-{y}_i(t)\|\leq2\sum_{i=1}^n\|{y}_i(t)-{\bar{y}}(t)\|$ and ${\gamma(t)}\leq 1$, it implies inequality $(\ref{eq14})$. \QEDA


%



\ifCLASSOPTIONcaptionsoff
  \newpage
\fi



%

\end{document}